\documentclass[11pt,letterpaper,reqno]{amsart}

\usepackage{euscript}
\usepackage{amsthm}
\usepackage{amscd}
\usepackage{amsmath}
\usepackage{graphicx}
\usepackage{color}

\usepackage{hyperref}

\usepackage[T1]{fontenc}
\usepackage{textcomp}
\usepackage{lmodern}
\usepackage{mathabx} 

\usepackage{microtype}

\usepackage[text={5.5in, 8.5in}, centering]{geometry}

\numberwithin{equation}{section}
\newtheorem{theorem}{Theorem}
\newtheorem{proposition}{Proposition}[section]

\newtheorem{corollary}[proposition]{Corollary}
\newtheorem{lemma}[proposition]{Lemma}

\newtheorem*{main}{Main result}

\newtheorem*{definition}{Definition}

\theoremstyle{remark}
\newtheorem*{remark}{Remark}

\newcommand{\Z}{\mathbb{Z}}

\newcommand{\R}{\mathbb{R}}

\newcommand{\N}{\mathbb{N}}

\newcommand{\T}{\mathbb{T}}

\newcommand{\bA}{\mathbb{A}}
\newcommand{\bE}{\mathbb{E}}

\DeclareMathOperator*{\argmin}{arg\,min}
\DeclareMathOperator{\supp}{supp}
\renewcommand{\mod}{\,\mathrm{mod}\,}

\newcommand{\Apm}{\mathbf{AP}^-}
\newcommand{\App}{\mathbf{AP}^+}

\newcommand{\chphi}{\widecheck{\varphi}}
\newcommand{\chpsi}{\widecheck{\psi}}
\newcommand{\chK}{\widecheck{K}}

\newcommand{\barN}{{\bar{N}}}

\begin{document}

\title[Convergence rate of random Hamilton-Jacobi equations]{Exponential convergence of solutions for random Hamilton-Jacobi equations}
\author{Renato Iturriaga}
\address{CIMAT, Guanajuato, Mexico}
\email{renato@cimat.mx}
\author{Konstantin Khanin}
\address{University of Toronto}
\email{khanin@math.toronto.edu}
\author{Ke Zhang}
\address{University of Toronto}
\email{kzhang@math.toronto.edu}


\begin{abstract}
	We show that for a family of randomly kicked Hamiton-Jacobi equations on the torus, almost surely, the solution of an initial value problem converges exponentially fast to the unique  stationary solution. Combined with the results in \cite{IK03} and \cite{KZ12}, this completes the program started in \cite{EKMS00} for the multi-dimensional setting.  
\end{abstract}

\maketitle

\section{Introduction}

\label{sec:introduction}

We consider the randomly forced Hamilton-Jacobi equation on the $d$ dimensional torus
\begin{equation}
	\label{eq:HJ}
	\partial_t \psi(x,t) + \frac12\left(  \nabla \psi(x,t) + b \right)^2  + F^\omega(x,t) = 0, \quad x \in \T^d = (\R/\Z)^d,
\end{equation}
where $b\in \R^d$, $\nabla$ stands for gradient in $x$, and $F^\omega$ is a random potential.  
By writing $u(x,t) = \nabla \psi(x,t)$, we obtain the stochastic Burgers equation 
\begin{equation}
	\label{eq:burgers}
	\partial_t u + (u\cdot \nabla)u = f^\omega(y,t), \quad y\in \R^d, t\in \R,
\end{equation}
$f^\omega(y,t) = - \nabla F^\omega(y,t)$ with the condition $\int u(x,t) dx = b$. This is one of the motivations of our study. On the other hand, \eqref{eq:HJ}  is a particular example of the more general Hamilton-Jacobi equation
\begin{equation}
	\label{eq:gen-HJ}
	\partial_t \psi + H^\omega(x, \nabla \psi) = 0
\end{equation}
where $H^\omega(x,p)$ is strictly convex and superlinear in $p$ (called the Tonelli Hamiltonians). Many of our results can be generalized to \eqref{eq:gen-HJ}, but we will restrict to \eqref{eq:HJ} for simplicity. 

We are interested in two types of random potentials. In  \cite{EKMS00}, the authors consider the dimension $d=1$, with the ``white noise potential''
\begin{equation}
	\label{eq:whitenoise}
	F^\omega(y,t) = \sum_{i=1}^M F_i(y,t) = \sum_{i=1}^M F_i(y) \dot{W}_i(t),
\end{equation}
where $F_i:\T^d \to \R$ are smooth functions, and $\dot{W}_i$ are independent white noises. It is shown that finite time solutions of \eqref{eq:HJ} converges exponentially fast to a unique stationary solution. In this paper, we generalize this result to arbitrary dimensions, for a related ``kicked'' model. 

The ``kicked force'' model was introduced in \cite{IK03}, with
\begin{equation}
	\label{eq:kicked} F^\omega(y,t) = \sum_{j\in \Z}F^\omega_j(y)\delta(t-j),
\end{equation}
where $F^\omega_j$ is an i.i.d. sequence of potentials, and $\delta(\cdot)$ is the delta function. We focus on the ``kicked'' potential \eqref{eq:kicked} as it is simpler, but retains most of the features of the system.

The system \eqref{eq:gen-HJ} does not admit classical solutions in general, and the solution is interpreted using the Lax-Oleinik variational principle. There is a semi-group family of operators (see \eqref{eq:lax-bk})
\[
	K_{s,t}^{\omega,b}: C(\T^d) \to C(\T^d),  
\]
such that the function $\psi(x,\tau) = K_{s,\tau}^{\omega,b}\varphi(x)$, $s \le \tau \le t$ is the solution to \eqref{eq:HJ} on the interval $[s,t]$ with the initial condition $\psi(x,s) = \varphi(x)$. 

It is shown in \cite{IK03} that under suitable conditions on the kicked force, almost surely, the system \eqref{eq:HJ} admits a unique solution $\psi^-_\omega(x,t)$ (up to an additive constant) on the interval $(-\infty, \infty)$. Let us denote 
\[
	\| \psi\|_* = \min_{C \in \R} \sup_{x \in \T^d}\|\psi(x)- C\|,
\]
which is the suitable semi-norm for measuring convergence up to an additive constant. Then  any solution on $[s,t]$ converges to $\psi^-_\omega$ as $s \to -\infty$, uniformly over all initial conditions in the semi-norm $\|\cdot \|_*$:
\[
	\lim_{s \to - \infty} \sup_{\varphi \in C(\T^d)} \| K_{s,t}^{\omega,b}\varphi(x) - \psi^-_\omega(x,t)\|_* = 0.
\]

Our main result is that the above convergence is exponentially fast. 
\begin{main}
There exists a (non-random) $\lambda>0$ such that, almost surely,
\[
	\limsup_{s \to -\infty} \frac{1}{|s|} \log \left(   \sup_{\varphi \in C(\T^d)} \| K_{s,t}^{\omega,b}\varphi(x) - \psi^-_\omega(x,t)\|_*  \right)
	< -\lambda, 
\]
see Theorem~\ref{thm:exp-conv}. 
\end{main}

\begin{remark}
Exponential convergence is also known to hold in the viscous equation
\[
	\partial_t \psi + \frac12 (\nabla \psi)^2 + F^\omega = \nu \Delta \psi
\]
(see \cite{Sinai1991}). However, in this case the a priori convergence rate $\lambda(\nu)\to 0$ as $\nu \to 0$. Since our result provides a non-zero lower bound on convergence rate when $\nu =0$, it is an interesting question whether a uniform rate of convergence exists for the viscous equation. 
\end{remark}

The \emph{a priori} convergence rate of the Lax-Oleinik semi-group is only polynomial in time, as evidence in the case when there is no force, i.e. $F^\omega = 0$.  In the case when the force is non-random, exponential convergence is true when the Aubry set consists of finitely many hyperbolic periodic orbits or fixed points (\cite{IS2009}). According to a famous conjecture of Ma\~ne, this condition holds for a generic force (\cite{Mane1997}), however this conjecture is only proven when $d=2$ among $C^2$ forces (\cite{CFR2015}, \cite{Contreras2014a}). 

In some sense,  \cite{KZ12} proves a random version of Ma\~ne's conjecture. In the random case, the role of the Aubry set is taken by the the globally minimizing orbit, and it is shown that this orbit is non-uniformly hyperbolic under the  random Euler-Lagrange flow. Conceptually, this hyperbolicity then allows the exponential convergence. However, this is quite delicate. To illustrate, let us outline the proof in the uniform hyperbolic case:
\begin{itemize}
 \item \emph{(Step 1)} Consider a solution $K_{-T, 0}^{\omega, b}\varphi$ that is sufficiently close to the stationary solution $\psi^-(\cdot, 0)$, this is the case since we know the solution $K_{-T, 0}^{\omega, b} \varphi \to \psi^-(\cdot, 0)$, albeit without any rate estimates. 
 \item \emph{(Step 2)} Show that the associated finite time minimizers is close to the Aubry set when $t \in [-2T/3, -T/3]$. By hyperbolic theory, any orbit that stays in a neighborhood of an hyperbolic orbit for time $T/3$ must be exponentially close to it at some point.
 \item \emph{(Step 3)} Finite time minimizer being exponentially close to the Aubry set implies the solution is in fact exponentially close to $\psi^-$. 
\end{itemize}

In the non-uniform hyperbolic case, Step 2 fails, because a non-uniform hyperbolic orbit only influence nearby orbits in a \emph{random} neighborhood whose size changes from iterate to iterate. We are forced to devise a much more involved procedure: 
\begin{enumerate}
 \item \emph{(Step A)} Reduce the problem to a local one, where we only study the solution in a small (random) neighborhood of the global minimizer. 
 \item \emph{(Step B)} Consider a solution $K_{-T, 0}^{\omega, b}\varphi$ is $\delta$-close to the stationary solution $\psi^-(\cdot, 0)$ locally. Use a combination of variational and non-uniform hyperbolic theory to show that the finite time minimizer is $\delta^q$-close to the global minimizer at some time, where $q > 1$. This step can only be done up to an exponentially small error. 
 \item \emph{(Step C)} Use step B to show the solution $K_{-T, 0}^{\omega, b}\varphi$ is $\delta^q$-close to the stationary solution. Feed the new estimates into  Step B, and repeatedly upgrade until $\delta$ is exponentially small. 
\end{enumerate}

We now present the outline of the paper. We formulate our assumptions and main result in Section~\ref{sec:stat}. Basic properties of the viscosity solutions and stationary solutions are introduced in Sections \ref{sec:visc} and \ref{sec:prop-visc}. In Section~\ref{sec:local-conv}, we reduce the main result to its local version, as outlined in Step A. This is Proposition~\ref{prop:local-conv}. 

In Section~\ref{sec:upgrade}, we describe the upgrade procedure outlined in Step C. Step B is formulated in Proposition~\ref{prop:dt-localize}, and the proof is postponed to Sections \ref{sec:finite-time} and \ref{sec:hyper-theory}.

\section{Statement of the main result}
\label{sec:stat}

Consider the kicked potentials \eqref{eq:kicked}, where the random potentials $F^\omega_j$ are chosen independently from a distribution $P \in \mathcal{P}(C^{2+\alpha}(\T^d))$, with $0 < \alpha \le 1$. 

Given an absolutely continuous curve $\zeta:[s,t]\to \T^d$, we define the action of $\zeta$ to be 
\[
	\bA^{\omega, b}(\zeta) =  \int_s^t \frac12 \left(  \dot{\zeta}^2(\tau) - b \cdot \dot{\zeta} \right) d\tau - \sum_{s \le j < t} F^\omega_j(\zeta(j)).
\]
In other words, when $s,t$ are integers, we include the kick at time $s$, but not at time $t$. For $0 < s < t \in \R$,  and $x, x' \in \T^d$, the action function is 
\begin{equation}
	\label{eq:action}
	A_{s,t}^{\omega, b}(x, x') = \inf_{\zeta(s) = x, \, \zeta(t) = x'} \bA^{\omega, b}(\zeta),
\end{equation}
where $\zeta$ is absolutely continuous. The action function is Lipshitz in both variables. 

The backward Lax-Oleinik operator $K^{\omega,b}_{s,t} : C(\T^d) \to C(\T^d)$ is defined as 
\begin{equation}
	\label{eq:lax-bk}
	K^{\omega,b}_{s,t}\varphi(x) = \min_{y \in \T^d} \left\{  \varphi(y) + A_{s,t}^{\omega,b}(y, x)  \right\}.
\end{equation}
We take \eqref{eq:lax-bk} as the definition of our solution on $[s,t]$ with initial conditon $\varphi(x)$. 
Due to the fact that $F^\omega(x,t)$ vanishes at non-integer times,  $K_{s,t}^{\omega,b}$ is completely determined by its value at integer times. In the sequel we consider only $s=m, t =n \in \Z$. The operators satisfies a semi-group property: for $s <  t < u $, 
\[
	K^{\omega,b}_{t, u} K^{\omega,b}_{s,t}\varphi(x) = K^{\omega,b}_{s,u}\varphi(x)
\]

We now state the conditions on the random potentials. The following assumptions are introduced in \cite{IK03}, which guarantees the uniqueness of the stationary solution.  
\begin{itemize}
	\item[] \emph{Assumption 1.} For any $y\in\T^d$, there exists $G_y\in \supp P$ s.t. $G_y$ has a maximum at $y$ and that there exists $\delta>0$ such that 
	$$ G_y(y)-G(x)\ge \delta|y-x|^2. $$
	\item[] \emph{Assumption 2.} $0\in \supp P$. 
	\item[] \emph{Assumption 3.} There exists $G\in \supp P$ such that $G$ has a unique maximum. 
\end{itemize}
The following is proved in \cite{IK03} under the weaker assumption that $F_j^\omega\in C^1(\T^d)$: 
\begin{proposition}\label{prop:backward-min}\cite{IK03}
\begin{enumerate}
	\item Assume that assumption 1 or 2 holds.  For a.e. $\omega\in \Omega$, we have the following statements.
	\begin{enumerate}
		\item There exists a Lipshitz function $\psi^-(x,n)$, $n \in \Z$, such that for any $m<n$,
		$$ K_{m,n}^{\omega,b}\psi^-(x,m)=\psi^-(x,n).$$
		\item For any $n \in \Z$, we have
		$$ \lim_{m\to -\infty}\sup_{\varphi \in C(\T)}\|K^{\omega,b}_{m,n}\varphi(x)-\psi^-(x,n)\|_* =0. $$
	\end{enumerate}
	\item Assume that assumption 3 holds. Then the conclusions for the first case hold for $b=0$.  
\end{enumerate}
\end{proposition}

We now restrict to a specific family of kicked potentials. The following assumption is introduced in \cite{IK03}.
\begin{itemize}
	\item[] \emph{Assumption 4.} Assume that 
	\begin{equation}
		\label{eq:F-omega}
		F^\omega_j(x)=\sum_{i=1}^M \xi_j^i(\omega)F_i(x),    
	\end{equation}
	where $F_i:\T^d\to \R$ are smooth non-random functions, and the vectors $\xi_j(\omega)=(\xi_j^i(\omega))_{i=1}^M$ is an i.i.d sequence of vectors in $\R^M$ with an absolutely continuous distribution. 
\end{itemize}
In \cite{KZ12}, a stronger assumption is used to obtain information on the stationary solutions and the global minimizer. These additional structures provides the mechanism for exponential convergence.  Let $\rho: \R^m \to \R$ be the density of $\xi_j$. 
\begin{itemize}
	\item[] \emph{Assumption 5.}  Suppose assumption 4 holds, and in addition:
	\begin{itemize}
		\item 
		\[
			\bE(|\xi_j|)=\int_{\R^M} |c|\rho(c)dc < \infty.
		\]
		\item For every $1\le i \le M$, there exists non-negative functions $\rho_i \in L^\infty(\R)$ and $\hat{\rho}_i\in L^1(\R^{M-1})$ such that 
		$$ \rho(c)\le \rho_i(c_i) \hat{\rho}_i(\hat{c}), $$
		where $c=(c_1, \cdots, c_M)$, $\hat{c}_i =(c_1, \cdots, c_{i-1}, c_{i+1}, \cdots, c_M)$. 
	\end{itemize}
\end{itemize}
Assumption 5 is rather mild. We only need to avoid the case that $\rho$ is degenerate in some directions. In particular, it is satisfied if $\xi^1_j, \cdots, \xi^M_j$ are i.i.d. random variables with bounded densities and finite mean. 

We now state the main theorem of this paper.
\begin{theorem}
	\label{thm:exp-conv}
	\begin{enumerate}
		\item   Assume that assumption 5 and one of assumption 1 or 2 hold. Assume in addition that the mapping 
		\begin{equation}
			\label{eq:FM}
			(F_1, \cdots, F_M): \T^d \to \R^M
		\end{equation} 
		is an embedding. For $b \in \R^d$, let $\psi^-_\omega$ be the unique stationary solution in Proposition~\ref{prop:backward-min}.  Then there exists a (non-random) $\lambda >0$  and a random variables $N(\omega)>0$ such that almost surely, for all $N > N(\omega)$, 
		\[
			\sup_{\varphi \in C(\T^d)} \| K_{m,n}^{\omega,b}\varphi - \psi^-_\omega(\cdot ,n)\|_*   
			\le  e^{-\lambda N}. 
		\]
		\item Assume that assumption 3 and 5 hold. Then the same conclusions hold for $b=0$. 
	\end{enumerate}
\end{theorem}

\section{Viscosity solutions and the global minimizer}
\label{sec:visc}

Let $I \subset \R$ be an interval. An absolutely continuous curve $\gamma: I \to \T^d$ is called a minimizer if for each interval $[s,t] \subset I$, we have $A^{\omega,b}_{s,t}(\gamma(s), \gamma(t)) = \bA^{\omega,b}(\gamma|_{[s,t]})$. In particular, $\gamma$ is called a forward minimizer if $I = (-\infty, t_0]$, a backward minimizer if $I = [s_0, \infty)$ and a global minimizer if $I = (-\infty, \infty)$. 

Due to the kicked nature of the potential, a minimizer is always linear between integer times. Then any minimizer  $\gamma:[m, n] \to \infty$ is completely determined by the sequence
\begin{equation}
	\label{eq:discrete}
	x_j = \gamma(j), \quad v_j = \dot{\gamma}(j-),\quad  m+1 \le j \le n. 
\end{equation}
The underlying dynamics for the minimizers is given by  family of maps $\Phi_j^\omega:\T^d\times \R^d \to \T^d \times \R^d$
\begin{equation}
	\label{eq:twistmaps}
	\Phi_j^\omega:
	\begin{bmatrix}
		x \\ v
	\end{bmatrix}
	\mapsto
	\begin{bmatrix}
		x + v - \nabla F_j^\omega(x)  \mod \Z^d. 
		\\
		v - \nabla F_j^\omega(x)                 
	\end{bmatrix},
\end{equation}
The maps belong to the so-called \emph{standard family}, and are examples of symplectic, exact and monotonically twist diffeomorphisms.  For $m,n\in \Z$, $m<n$, denote
$$ \Phi^\omega_{m, n}(x,v)=\Phi^\omega_{n-1}\circ \cdots \circ \Phi^\omega_{m}(x,v). $$
The (full) orbit of a vector $(x_n,v_n)$ is given by the sequence
\[
	(x_j, v_j) = \Phi_{n,j}^\omega(x_n, v_n), \quad j >n, \quad (x_j, v_j) = (\Phi_{j,n}^\omega)^{-1}(x_n, v_n), \quad j < n. 
\] 
If  $\gamma:[m,n]\to \T^d$ is a minimizer, then  $(x_j, v_j)$ defined in \eqref{eq:discrete} is an orbit, namely
\[
	\Phi_{j,k}^\omega (x_j, v_j) = (x_k, v_k), \quad m +1 \le j < k \le n. 
\]
In this case, we extend the sequence to $(x_m, v_m) = (\Phi_m)^{-1}(x_{m+1}, v_{m+1})$ and call $(x_j, v_j)_{j = m}^n$ a minimizer. 

The viscosity solution and the minimizers are linked by the following lemma:
\begin{lemma}[\cite{KZ12}, Lemma 3.2]\label{lem:min}
\begin{enumerate}
	\item For $\varphi \in C(\T^d)$  and $m < n \in \Z$, for each $x \in \T^d$ there exists a minimizer $(x_j, v_j)_{j =m}^n$ such that $x_n =x$, and 
	\begin{equation}
		\label{eq:finite-min}
		K^{\omega,b}_{m,n}\varphi(x_n) = \varphi(x_m) + A_{m,n}^{\omega,b}(x_m, x_n). 
	\end{equation}
	Moreover, the minimizer is unique if $\psi(x)  = K^{\omega,b}_{m,n}\varphi(\cdot)$ is differentiable at $x$, and in this case $v_n = \nabla \psi(x) + b$. 
	\item Suppose $\psi^-_\omega(x,n)$ is the stationary solution. Then at every $x\in \T^d$ and $n \in \Z$, there exists a backward minimizer $(x_j, v_j)_{j = -\infty}^n$ such that $x_n = x$
	\[
		K^{\omega, b}_{m,n}\psi^-_\omega(x_n, n) = \psi^-_\omega(x_m, m) + A_{m,n}^{\omega, b}(x_m, x_n), \quad m < n. 
	\]
	Moreover, the minimizer is unique if $\psi^-_\omega(\cdot , n)$ is differentiable at $x$, and in this case $v_n = \nabla \psi^-_\omega(x, n) + b$. 
\end{enumerate}
\end{lemma}
In case (1) we call $(x_j, v_j)_{j=m}^n$ a minimizer for $K_{m,n}^{\omega, b}\varphi(x_0)$, and in case (2) the orbit $(x_j, v_j)_{j=-\infty^n}$ is called a minimizer for $\psi^-(x_n, n)$. 

The forward minimizer is linked to the forward operator $\chK_{s,t}^{\omega, b}$, defined as
\[
	\chK_{s,t}^{\omega,b}\varphi(x) = \sup_{y \in \T^d} \left\{  \varphi(y) - A_{s,t}^{\omega,b}(x, y) \right\}.
\]
Analog of Proposition~\ref{prop:backward-min} and Lemma~\ref{lem:min} hold, which we summarize below.
\begin{itemize}
	\item For every $b \in \R^d$, almost surely, there exists a unique Lipshitz function $\psi^+_\omega(x, m)$, $m \in \Z$, such that 
	\[
		\chK^{\omega,b}_{m,n}\psi^+(x, n) = \psi^+(x,m), \quad  m < n. 
	\]
	\item For each $\chK^{\omega,b}_{m,n}\varphi(x)$ there exists a minimizer $(x_j, v_j)_{j=m}^n$ such that $x_m = x$ and 
	\[
		\chK^{\omega, b}_{m,n}\varphi(x_m) = \varphi(x_n) - A_{m,n}^{\omega,b}(x_m, x_x).
	\]
	When $\psi = \chK^{\omega, b}_{m,n}\varphi$ is differentiable at $x$ we have $v_m = \nabla \psi(x) + b$. 
	\item For each $x \in \T^d$, and $m \in \Z$, there exists a forward minimizer $(x_j, v_j)_{j = m}^\infty$ such that $x_m = x$,
	\[
		\chK^{\omega,b}_{m,n}\psi^+_\omega(x_m, m) = \psi^+_\omega(x_n, n) - A_{m,n}^{\omega, b}(x_m, x_n), \quad m<n,
	\]
	and $v_m = \nabla \psi^+_\omega(x, m)$ if $\psi^+(\cdot, m)$ is differentiable at $x$. 
\end{itemize}
\vskip .1in

The global minimizer is characterized by both $\psi^-_\omega$ and $\psi^+_\omega$. 
\begin{proposition}\label{prop:glob}
Assume that Assumption 4 holds, and one of Assumptions 1 and 2 holds. Assume in addition, the map \eqref{eq:FM} is an embedding. Then for $b\in \R^d$, almost every $\omega$, there exists a unique global minimizer $(x_j^\omega, v_j^\omega)_{j \in \Z}$. For each $j \in \Z$,  $x_j^\omega$ is the unique $x\in \T^d$ reaching the minimum in
\begin{equation}
	\label{eq:min-var}
	\min_x \{ \psi^-_\omega(x,j) - \psi^+_\omega(x, j)\}.
\end{equation}
Moreover, $\psi^\pm_\omega(\cdot, j)$ are both differentiable at $x_j$, and $v_j = \nabla \psi^-_\omega(x, j) + b = \nabla \psi^+_\omega(x, j) + b$. 
\end{proposition}
The function 
\begin{equation}
	\label{eq:psi-inf}
	Q^\infty_\omega(x,j)  := \psi^-_\omega(x,j) - \psi^+_\omega(x, j)
\end{equation}
will serve an important purpose for the discussions below. 

The random potentials $F_j^\omega$ are generated by a stationary random process, so there exists a measure preserving transformation $\theta$ on the probability space $\Omega$ satisfying
\begin{equation}
	\label{eq:time-shift}
	F^\omega_{n+m}(x) = F^{\theta^m \omega}_{n}(x).   
\end{equation}
The family of maps $\Phi^\omega_j$ then defines a non-random transformation 
\[
	\hat{\Phi}(x,v, \omega) = (\Phi_0(x,v), \theta \omega)
\]
on the space $\T^d \times \R^d \times \Omega$. Then from Proposition~\ref{prop:glob},
\[
	(x_0^{\theta\omega}, v_0^{\theta \omega}) = \Phi_0^\omega(x_0^\omega, v_0^\omega)
\]
and the probability measure
\[
	\nu(d(x,v), d\omega) = \delta_{(x_0^\omega, v_0^\omega)}P(d\omega)
\]
is invariant and ergodic under $\hat{\Phi}$. The map $D\Phi_0^\omega: \T^d \times \R^d \times \Omega \to Sp(d)$, where $Sp(d)$ is the group of all $2d \times 2d$ symplectic matrices, defines a cocycle over $\hat{\Phi}$. Under Assumption 5, its Lyapunov exponents $\lambda_1(\nu), \cdots, \lambda_{2d}(\nu)$ are well defined, and due to symplecticity, we have 
\[
	\lambda_1(\nu)\le \cdots \lambda_d(\nu) \le 0 \le \lambda_{d+1}(\nu) \le \cdots \le \lambda_{2d}(\nu),
\]
and $\lambda_i = -\lambda_{2d-i+1}$. 

There is a close relation between the non-degeneracy of the variational problem \eqref{eq:min-var}, and non-vanishing of the Lyapunov exponents for the associated cocycle. 

\begin{proposition}[\cite{KZ12}, Proposition 3.10]\label{prop:hyp}
Assume that assumption 5 and one of assumptions 1 or 2 holds. Assume in addition that the map \eqref{eq:FM} is an embedding. Then for all $b\in \R^d$, for a.e. $\omega$, the following hold.
\begin{enumerate}
	\item There exist $C(F,\rho), R(F,\rho)>0$ depending only on $F_1, \cdots F_M$ in \eqref{eq:FM} and the density $\rho$ of $\xi_j$, and a positive random variable $a(\omega)>0$ such that 
	\begin{equation}
		\label{eq:non-deg}
		Q_\omega^\infty(x, 0) - Q_\omega^\infty(x_0^\omega, 0) \ge  a(\omega)\|x-x_0^\omega\|^2, \quad \|x- x_0^\omega\| <  R(F, \rho), 
	\end{equation}
	with
	\begin{equation}
		\label{eq:moment-a}
		\bE (a(\omega)^{-\frac12}) < C(F, \rho). 
	\end{equation}
	\item The Lyapunov exponents of $\nu$ satisfy
	$$ \lambda_d(\nu)<0<\lambda_{d+1}(\nu).$$
\end{enumerate}
\end{proposition}

The second conclusion of Proposition~\ref{prop:hyp} implies the orbit $(x_j^\omega, v_j^\omega)$ for the sequence of maps $\Phi_j^\omega$ is non-uniformly hyperbolic. In particular, it follows that there exists local unstable and stable manifolds. It is shown in \cite{KZ12} that the graph of the gradient of the viscosity solutions locally coincide with the unstable and stable manifolds.

\begin{proposition}[\cite{KZ12}, Theorem 6.1]\label{prop:unst-stab}
Under the same assumptions as Proposition~\ref{prop:hyp}, for each $\epsilon>0$,  there exists positive random variables $r(\omega)>0$, $C(\omega)>1$, such that the following hold almost surely.
\begin{enumerate}
	\item There exists $C^1$ embedded submanifolds $W^u(x_0^\omega, v_0^\omega)$ and $W^s(x_0^\omega, v_0^\omega)$, such that 
	\[
		(x, \nabla \psi^-_\omega(x, 0) + b) \in W^u(x_0^\omega, v_0^\omega), \quad 
		(x, \nabla \psi^+_\omega(x,0)+ b) \in W^s(x_0^\omega, v_0^\omega)
	\]
	for all $\|x - x_0^\omega\| < r(\omega)$. 
	\item For every $\|x - x_0^\omega\| < r(\omega)$, let $(x_j^-, v_j^-)_{j \le 0}$ and $(x_j^+, v_j^+)_{j \ge 0}$ be the backward and forward minimizers satisfying $x_0^\pm = x$.  Then
	\[
		\| (x_j^\omega, v_j^\omega) - (x_j^-, v_j^-) \| \le C(\omega) e^{-\lambda'|j|}, \quad j \le 0,
	\]
	\[
		\| (x_j^\omega, v_j^\omega) - (x_j^+, v_j^+) \| \le C(\omega) e^{-\lambda'|j|}, \quad j \ge 0,
	\]
	where $\lambda' = \lambda_{d+1}(\nu) - \epsilon$. 
\end{enumerate}
\end{proposition}
\begin{remark}
In Lemma~\ref{lem:tempered-arK} we will show that the random variables $r(\omega), C(\omega)$ in item (2) can be chosen to satisfy an additional \emph{tempered} property. 
\end{remark}


\section{Properties of the viscosity solutions}
\label{sec:prop-visc}

\subsection{Semi-concavity}

Given $C>0$, we say that a function $f: \R^d \to \R$ is $C$ semi-concave if for any $x\in \R^d$, there exists a linear form $l_x: \R^d \to \R$ such that
\[
	f(y) - f(x) \le l_x(y-x) + C\|y-x\|^2, \quad y \in \R^d. 
\] 
A function $\varphi: \T^d \to \R$ is called $C$ semi-concave if it is $C$ semi-concave as a function lifted to $\R^d$. The linear form $l_x$ is called a subdifferential at $x$. If $\varphi$ is differentiable at $x\in \T^d$, then the subdifferential $l_x$ is unique and is equal to $d \varphi(x)$. A semi-concave function is Lipshitz. 

\begin{lemma}[\cite{fat08}, Proposition 4.7.3]\label{lem:Lipshitz}
If $\varphi$ is  continuous and $C$ semi-concave on $\T^d$, then $\varphi$ is $2C\sqrt{d}$-Lipshitz. 
\end{lemma}
\begin{lemma}[\cite{fat08}]\label{lem:grad-lip}
Suppose both $\varphi_1$ and $-\chphi_2$ is $C$ semi-concave, then over the set $\argmin_x\{\varphi_1(x) - \chphi_2(x)\}$,  $\varphi_1, \chphi_2$ are differentiable, $\nabla \psi_1(x) = \nabla \psi_2(x)$, and $\nabla \varphi_1(x)$ is $6C$-Lipshitz over the set $\argmin_x\{\varphi_1(x) - \chphi_2(x)\}$. 
\end{lemma}

Let $K_0(\omega)=\|F_0^\omega\|_{C^{2+\alpha}}+1$ and $K(\omega) = 2\sqrt{d}(K_0(\omega)+1)$. The action function $A_{m,n}^{\omega,b}$ has the following properties. 

\begin{lemma}[\cite{KZ12}, Lemma 3.2]\label{lem:prop-act}
\begin{enumerate}
	\item The function $A_{m,n}^{\omega,b}(x, x')$ is $1-$semi-concave in the second component, and is $K(\theta^m\omega)-$semi-concave in the first component. Here $\theta: \Omega \to \Omega$ is the time-shift, see \eqref{eq:time-shift}. 
	\item For any $\varphi \in C(\T^d)$, and $m < n \in \Z$, the function $K_{m,n}^{\omega, b}\varphi(x)$ is $1$ semi-concave, and $-\chK_{m,n}^{\omega, b}\varphi(x)$ is $K_0(\theta^m\omega)$ semi-concave. Either function, as well as the sum of the two functions, are $K(\theta^m\omega)$ Lipshitz. 
	\item For $n \in \Z$, the functions $\psi^-_\omega(\cdot, n)$ is $1$ semi-concave, and $-\psi^+_\omega(\cdot, n)$ is $K_0(\theta^n \omega)$ semi-concave,  Either function, as well as the sum of the two functions, are $K(\theta^m\omega)$ Lipshitz.  
\end{enumerate}
\end{lemma}

We first state two lemmas concerning the properties of the Lax-Oleinik semigroup, the goal is to obtain Lemma~\ref{lem:lipsht-est}, which is a version of Mather's graph theorem (\cite{Mather93}). 
\begin{lemma}\label{lem:ineq-forward-backward}
For any $x \in \T^d$, $m  < n$, $\varphi \in C(\T^d)$, we have 
\[
	\chK_{m,n}^{\omega,b}\left(   K_{m,n}^{\omega,b}\varphi \right)(x) \le \varphi(x)
\]
\end{lemma}
\begin{proof}
	For any $x,y \in \T^d$, 
	\[
		K_{m,n}^{\omega,b}\varphi(y) \le \varphi(x) + A_{m,n}^{\omega,b}(x,y),
	\]
	then
	\[
		\chK_{m,n}^{\omega,b}\left(  K_{m,n}^{\omega,b}\varphi \right)(x) = \max_{y\in \T^d}\left\{  K_{m,n}^{\omega,b}(y) - A_{m,n}^{\omega,b}(x,y)  \right\} \le \varphi(x). 
	\]
\end{proof}
\begin{lemma}\label{lem:forward-backward}
Suppose $m < n $. Let $(x_j, v_j)_{j=m}^n$ be a minimizer for $K_{m,n}^{\omega,b}\varphi$ in the sense of \eqref{eq:finite-min}. Then for each $m\le j \le n$, we have 
\[
	K_{m,j}^{\omega,j}\varphi(x) \ge \chK_{j,n}^{\omega,b}\left(  K_{m,n}^{\omega,b}\varphi \right)(x), \quad \forall x \in \T^d,
\]
and 
\[
	K_{m,j}^{\omega,j}\varphi(x_j) = \chK_{j,n}^{\omega,b}\left(  K_{m,n}^{\omega,b}\varphi \right)(x_j).	
\]
\end{lemma}
\begin{proof}
	By definition $K_{nn}^{\omega, b}\varphi = \varphi$, so the case $j = n$ is trivial. 
	For $m \le j < n$, we have 
	\[
		K_{m,n}^{\omega,b}\varphi(x_n) = K_{m,j}^{\omega, b}\varphi(x_j) + A_{j,n}^{\omega, b}(x_j, x_n).
	\]
	Then 
	\[
		K_{m,j}^{\omega, b}\varphi(x_j) = K_{m,n}^{\omega,b}\varphi(x_n)  - A_{m,n}^{\omega, b}(x_j, x_n) \le \chK_{j,n}^{\omega,b}\left(  K_{m,n}^{\omega,b}\varphi \right)(x_j).
	\]
	On the other hand, apply Lemma~\ref{lem:ineq-forward-backward} to $K_{m,j}^{\omega,b}\varphi$ and $j < n$ yields
	\[
		\chK_{j,n}^{\omega,b}\left(  K_{m,n}^{\omega,b}\varphi \right)(x) = \chK_{j,n}^{\omega,b}\left( K_{j,n}^{\omega,b} K_{m,j}^{\omega,b}\varphi \right)(x) \le K_{m,j}^{\omega,b} \varphi(x). 
	\]
	The lemma follows. 
\end{proof}

As a result, A minimizer of a backward solution is also a minimizer of the forward-backward solution. 
\begin{corollary}\label{cor:two-way-minimizer}
Let $(x_j, v_j)_{j=m}^n$ be a minimizer for $K_{m,n}^{\omega,b}\varphi$, then it is also a minimizer for $\chK_{m,n}^{\omega, b} \left( K_{m,n}^{\omega,b}\varphi \right)$. 
\end{corollary}
\begin{proof}
	Using the calculations in the proof of Lemma~\ref{lem:forward-backward}, we get 
	\[
		\chK_{j, n}^{\omega, b}\left(  K_{m,n}^{\omega,b}\varphi \right)(x_j) = K_{m, j}^{\omega, b} \varphi(x_j) = K_{m,n}^{\omega, b}\varphi(x_n) - A_{j, n}^{\omega, b}(x_j, x_n)
	\]
	for all $m \le j \le n$. The corollary follows. 
\end{proof}

The following lemma provides a Lipshitz estimate for the velocity of the minimizer in the interior of the time interval. 

\begin{lemma}\label{lem:lipsht-est}
Suppose $m < n $ with $n-m \ge 2$. Let $(x_j, v_j)_{j=m}^n$ and $(y_j, \eta_j)_{j=m}^n$ be two minimizers for $K_{m,n}^{\omega,b}\varphi$ in the sense of \eqref{eq:finite-min}. Then for all $m< j < n$, we have
\[
	\| v_j - \eta_j\| \le K(\theta^j \omega) \|x_j - y_j\|.
\]
The same conclusion hold, if $(x_j, v_j)$ and $(y_j, \eta_j)$ are minimizers for $\chK_{m,n}^{\omega,b}\varphi(x)$. 
\end{lemma}
\begin{proof}
	We apply Lemma~\ref{lem:forward-backward} to $(x_j, v_j)$ and $(y_j, \eta_j)$. Denote $\psi_1 = K_{m,j}^{\omega,j}\varphi$ and $\psi_2 = \chK_{j,n}^{\omega,b}\left(  K_{m,n}^{\omega,b}\varphi \right)$, since $j-m, n-j \ge 1$,  $\psi_1$ is $1$ semi-concave,  $-\psi_2$ is $K(\theta^j \omega)$ semi-concave. Since $x_j, y_j \in \argmin_x\{\psi_1(x) - \psi_2(x)\}$,  Lemma~\ref{lem:grad-lip} and Lemma~\ref{lem:min} implies
	\[
		\|v_j - \eta_j\| = \| \nabla \psi_1(x_j) - \nabla \psi_1(y_j) \| \le K(\theta^j \omega) \|x_j - y_j\|. 
	\]
\end{proof}

\subsection{Properties of the stationary solutions}

Recall that 
\[
	Q^\infty_\omega(x, n)  = \psi^-_\omega(x,n) - \psi^+_\omega(x,n),
\]
which takes its minimum at the global minimizer $x_n^\omega$. To simplify notations, we will drop the subscript $\omega$ from these functions when there is no confusion. 

This function $Q^\infty$ is very useful, as it can be used to measure the distance to the global minimizer. For all $ \|y - x_0^\omega\| < r(F)$, we have 
\begin{equation}
	\label{eq:metric}
	a(\omega) \|y - x_0^\omega\|^2 \le Q^\infty(x, 0) - Q^\infty(x_0^\omega, 0) \le K(\omega) \|y - x_0^\omega\|^2.
\end{equation}
Moreover, $Q^\omega$ is a Lyapunov function for infinite backward minimizers. Namely, if $(y_0, \eta_0) = (y_0, \nabla \psi^-_\omega(y_0, 0))$ is a backward minimizer, then for any $j<k\le 0$, we have 
\begin{equation}
	\label{eq:inf-lya}
	Q^\infty(y_j, j) - Q^\infty(x_j^\omega, j) \le Q^\infty(y_k, k) - Q^\infty(x_k^\omega, k).
\end{equation}
(See \cite{KZ12}, Lemma 7.2)

Let us also recall, for any $\lambda' < \lambda$, there exists functions $r(\omega)$, $C(\omega)>0$ such that for all backward minimizers $(y_n, \eta_n)_{n \le 0}$ such that 
\begin{equation}
	\label{eq:exp-backward}
	\|y_n - x_n^\omega\| \le C(\omega) \exp (- \lambda' |n|), \quad \text{ for } \|y_0 - x_0^\omega\|< r(\omega) \text{ and } n \le 0. 
\end{equation}
We will also use a process in non-uniform hyperbolicity known as tempering. 
\begin{lemma}[\cite{BP07}, Lemma 3.5.7]\label{lem:tempering}
Let $g(\omega)>1$ be a random variable satisfying $\bE(\log g(\omega)) < \infty$, then for any $\epsilon>0$, there exists $g^\epsilon(\omega)>g(\omega)$ such that 
\begin{equation}
  \label{eq:tem-kernel}
  	e^{-\epsilon} \le \frac{g^\epsilon(\omega)}{g^\epsilon(\theta \omega)} \le e^\epsilon.
\end{equation}
\end{lemma}
Let's call a random variable $g(\omega) > 0$ tempered if for any $\epsilon>0$, both $g, g^{-1}$ admits an upper bound satisfying \eqref{eq:tem-kernel}. Products and inverses of tempered random variables are still tempered. 

The random variables $a, K$ in \eqref{eq:metric} and $r, C$ in \eqref{eq:exp-backward} are tempered. 
\begin{lemma}\label{lem:tempered-arK}
For any $\epsilon > 0$, there exists random variables 
\[
	a^\epsilon(\omega) < a(\omega), \quad r^\epsilon(\omega) < r(\omega), \quad  K^\epsilon(\omega) >  K(\omega), \quad C^\epsilon(\omega) > C(\omega)
\]
such that 
\[
	e^{-\epsilon} \le \frac{a^\epsilon(\omega)}{a^\epsilon(\theta\omega)} 
	,  \frac{r^\epsilon(\omega)}{r^\epsilon(\theta\omega)},  \frac{K^\epsilon(\omega)}{K^\epsilon(\theta\omega)}, \frac{C^\epsilon(\omega)}{C^\omega(\theta \omega)} \le e^\epsilon. 
\]
\end{lemma}
\begin{proof}
	Lemma~\ref{lem:tempering} applies to $a(\omega)$ and $K(\omega)$ since $\bE (\log a^{-1}), \bE(\log K) < \infty$. The fact that $C(\omega)$ and $r(\omega)$ are tempered can be proven by adapting the proof Theorem 6.1 in  \cite{KZ12}. We now explain the adaptations required. 

	In \cite{KZ12}, there exists local linear coordinates $(s,u)$ with the formula $(y, \eta) = P_j(s, u)$ centered at the global minimizer $(x_j^\omega, v_j^\omega)$, with the estimates $\|DP_j\|, \|DP_j^{-1}\| \le K(\theta^j \omega)a^{-\frac12}(\theta^j \omega)$ (section 6 of \cite{KZ12}). Then it is shown that there exists a random variable $r_0(\omega)$ (called $r$ in that paper) such that any orbit $(y, \eta) \in \{(y, \nabla \psi^-_\omega(y, 0))\} \cap \{\|s\|, \|u\| < r_0\}$ must be contained in the stable manifold. $r$ is tempered because it is the product of tempered random variables.  Indeed, the following explicit formula was given in the Proof of Theorem 6.1, section 7 of \cite{KZ12}: 
	\[
		\bar{r} = \tilde{C}^{-3}K^{-9}a^6\kappa^2\rho^2, \quad r_0 = \bar{r}(\theta^{-1}\omega)K^3(\theta^{-1}\omega)a(\theta^{-1}\omega)\kappa^{\frac12}(\theta^{-1}\omega),
	\]
	where $K, a$ are the same as in this paper, and the fact that  $\rho, \tilde{C}$ are tempered is explained in Lemma 6.5 and Proposition 7.1 of \cite{KZ12}. We now convert to the variable $(y, \eta)$. Since the norm of the coordinate changes are tempered, there exists a tempered random variable  $r_1(\omega)$ such that any orbit contained in 
	\[
		\{(y, \nabla \psi^-_\omega(y, 0))\} \cap \{\|(y, \eta) - (x_0^\omega, v_0^\omega)\| < r_1(\omega)\} 
	\]
	must be contained in the unstable manifold of $(x_0^\omega, v_0^\omega)$.

	We now show the same conclusion holds on a neighborhood of the configuration space $\|y - x_0^\omega\|< r(\omega)$, with $r(\omega)$ tempered. Let $(y_0, \eta_0) = (y_0, \nabla \psi^-(y_0, 0))$ and let $(y_{-1}, \eta_{-1})$ be its backward image. According to Lemma~\ref{lem:grad-lip}, $\|\eta_{-1} - v_{-1}^\omega\| \le K(\theta^{-1}\omega)\|y_{-1} - x_{-1}^\omega\|$, as a result
	\[
		\|y_{-1} - x_{-1}^\omega\| < \frac{r_1}{1+K}((\theta^{-1}\omega)) \text{ implies } \|(y_{-1}, \eta_{-1}) - (x_{-1}^\omega, v_{-1}^\omega)\| < r_1(\theta^{-1}\omega). 
	\]
	Finally, using \eqref{eq:metric} and \eqref{eq:inf-lya}, we have 
	\[
		\begin{aligned}
			& \|y_{-1} - x_{-1}\| \le a^{\frac12}(\theta^{-1}\omega) \left( Q^\infty(y_{-1}, -1) - Q^\infty(x_{-1}, -1) \right)                                                \\
			& \le  a^{\frac12}(\theta^{-1}\omega) \left( Q^\infty(y_{0}, 0) - Q^\infty(x_{0}, 0) \right) \le a^{\frac12}(\theta^{-1}\omega) K^{\frac12}(\omega) \|y_0 - x_0\|, 
		\end{aligned}
	\]
	We obtain that 
	\[
		\|y_0 - x_0^\omega\| < \frac{r_1}{(1+K)(a^{-\frac12} \circ \theta^{-1}) K^{\frac12}} =: r
	\]
	implies $(y_1, \eta_1)$ is contained in the unstable manifold of $(x_{-1}^\omega, v_{-1}^\omega)$. $r$ is tempered as it is products of tempered random variables. 

	The fact that an orbit on the stable manifold of a non-uniformly hyperbolic orbit converge at the rate $C(\omega) e^{-\lambda n}$, and that the coefficient $C(\omega)$ is tempered is a standard result in non-uniform hyperbolicity, see for example \cite{BP07}. 
\end{proof}

We now use what we obtained to get an approximation for the stationary solutions. 

\begin{lemma}\label{lem:phi-minus-approx}
There exists $C_1^\epsilon(\omega)>0$, $e^{-\epsilon} \le C_1(\omega)/C_1(\theta^\omega) \le e^\epsilon$,  such that for all $\|y - x_0^\omega\| < r(\omega)$,  and $n < 0$, we have 
\[
	\left\|  \psi^-(y, 0) - \psi^-(x_n, n) - A_{n,0}^{\omega,b}(x_n, y)  \right\| \le C_1(\omega) e^{- (\lambda'-\epsilon)|n|}.
\]
We also have the forward version: for $n > 0$, 
\[
	\left\|  \psi^+(y, 0) - \psi^+(x_n, n) + A_{0,n}^{\omega,b}(y, x_n)  \right\| \le C_1(\omega) e^{- (\lambda'-\epsilon)|n|}.
\]
\end{lemma}
\begin{proof}
	We only prove the backward version. 
	By definition, 
	\[
		\psi^-(y,0)  \le \psi^-(x_n,n) + A_{n, 0}^{\omega,b}(x_n, y).
	\]
	On the other hand, let $(y_n, \eta_n)_{n\le 0}$ be the minimizer for $\psi^-_\omega(\psi, 0)$, then 
	\begin{align*}
		& \psi^-(y, 0) = \psi^-(y_n,n) + A_{n, 0}^{\omega,b}(y_n, y)                                                              \\ 
		& \ge \psi^-(x_n, n) + A_{n,0}^{\omega,b}(x_n, y) - K^\epsilon(\theta^n \omega)\|x_n - y_n\|                              \\
		& \ge \psi^-(x_n, n) + A_{n,0}^{\omega,b}(x_n, y) -  K^\epsilon( \omega) C^\epsilon(\omega) e^{-(\lambda'-\epsilon) |n|}, 
	\end{align*}
	where $C^\epsilon$ is from Lemma~\ref{lem:tempered-arK}. 
	The lemma follows by repeating the calculation with $\epsilon/2$, and  taking $C_1^\epsilon(\omega) = K^{\epsilon/2}(\omega) C^{\epsilon/2}(\omega)$. 
\end{proof}

\section{Reducing to local convergence}
\label{sec:local-conv}

In this section we reduce the main theorem to its local version. 
\begin{proposition}\label{prop:local-conv}
Under the same assumptions as Proposition~\ref{prop:hyp}, there exists $0 < \lambda < \lambda_{d+1}(\nu)$,  positive random variables $0<\tau(\omega)<r(\omega)$,  $D_0(\omega)>0$, and $N_0(\omega)>0$ such that for all $N >  N_0(\omega)$ and $\varphi \in C(\T)$, 
\[
	\sup_{\varphi \in C(\T)} \min_{C \in \R}\max_{\|y - x_0^\omega\| \le \tau(\omega)}  \left| K_{-N,0}^{\omega,b}\varphi(y) - \psi^-_\omega(y, 0)  - C \right |  \le D_0(\omega) e^{- \lambda N}. 
\]
\end{proposition}

The proof of Proposition~\ref{prop:local-conv} is given in the next section. Next, we have a localization result, which says any minimizer for $K_{-N, M}^{\omega,b}\varphi$ or $\psi^-(\cdot, M)$ go through the neighborhood $\{ \|x- x_0^\omega\| < \rho(\omega) \}$ at time $t =0$, when $N, M$ are large enough.

\begin{proposition}\label{prop:localize}
Under the same assumptions as Proposition~\ref{prop:hyp}, let $\tilde{\rho}(\omega)>0$ be  a positive random variable. Then  there exists $M_0 = M_0(\omega)\in \N$ depending on $\tilde{\rho}(\omega), \omega$ such that the following hold. 

\begin{enumerate}
	\item  For any $N, M \ge M_0(\omega)$, let $(y_n, \eta_n)$, $-N \le n \le M$ be a (backward) minimizer for $K_{-N, M}^{\omega, b}\varphi(y_{M})$. Then $\|y_0 - x_0^\omega\| < \tilde{\rho}(\omega)$.
	\item (The forward version) For any $N, M \ge M_0(\omega)$, let $(y_n, \eta_n)$, $-N \le n \le M$ be a (forward) minimizer for $\chK_{-N, M}^{\omega, b}\varphi(y_{-N})$. Then $\|y_0 - x_0^\omega\| < \tilde{\rho}(\omega)$.
\end{enumerate}
\end{proposition}

We prove Proposition~\ref{prop:localize} using the following lemma, stating that the Lax-Oleinik operators are weak contractions. 
\begin{lemma}[\cite{IK03}, Lemma 3]\label{lem:weak-contraction} 
For any $\varphi_1,\chphi_2 \in C(\T^d)$, we have 
\[
	\| K_{m,n}^{\omega,b} \varphi_1 - K_{m,n}^{\omega, b}\chphi_2 \|_* \le \|\varphi_1 - \chphi_2\|_*. 
\]
\end{lemma}

\begin{proof}[Proof of Proposition~\ref{prop:localize}]
	We only prove item (1), as the (2) can be proven in the same way. 
	We denote, for $-N \le m \le M$, 
	\[
		\psi^{-N,m}(x) = K_{-N, m}^{\omega,b}\varphi(x),
	\]
	and let $(y_n, \eta_n)$, $-N \le n \le M$ be a minimizer for $K_{-N, M}^{\omega,b}\varphi$. For each $-N < n < M$, Lemma~\ref{lem:ineq-forward-backward} and \ref{lem:forward-backward} implies
	\[
		\psi^{-N, n}(x) - \chK_{n, M}^{\omega,b}\left(  \psi^{-N, M} \right)(x) \ge 0, \quad x \in \T^d,
	\]
	\[
		\psi^{-N, n}(y_n) - \chK_{n, M}^{\omega,b} \left(   \psi^{-N, M}\right)(y_n) = 0,
	\]
	hence
	\begin{equation}
		\label{eq:yn-argmin}
		y_n \in \argmin_x \left\{  \psi^{-N, n}(x) - \chK_{n, M}^{\omega,b}\left(  \psi^{-N, M} \right)(x)  \right\}. 
	\end{equation}

	Recall that for $Q^\infty(x,0) = \psi^-(x,0) - \psi^+(x,0)$, we have $x_0^\omega$ is the unique minimum for $Q^\infty(\cdot, 0)$. Define 
	\[	
		\delta(\omega) = \min_{\|x - x_0^\omega\| \ge \tilde{\rho}(\omega)} Q^\infty(x) - Q^\infty(x_0^\omega). 
	\]
	By Proposition~\ref{prop:backward-min}, we can choose $M_0(\omega)$ large enough such that for $N, M \ge M_0(\omega)$, 
	\[
		\|\psi^{-N, 0} - \psi^-(\cdot, 0)\|_*, \, \| \chK_{0, M}^{\omega,b}\left(  \psi^{-N, M} \right) - \psi^+(\cdot, 0)\|_* < \delta(\omega)/4. 
	\]
	As a result, there exists a constant $C \in \R$ such that 
	\[
		\left\|  \left(  \psi^{-N, 0} - \chK_{0, M}^{\omega,b}\left(  \psi^{-N, M} \right) \right) - Q^\infty(x, 0) - C \right\| < \delta(\omega)/2. 
	\]
	It follows that the minimum in \eqref{eq:yn-argmin} is never reached outside of $\{\|x - x_0^\omega\|< \rho(\omega)\}$. We obtain  $\|y_0 - x_0^\omega\| < \tilde{\rho}(\omega)$.
\end{proof}

We now prove our main theorem assuming Proposition~\ref{prop:local-conv}. 
\begin{proof}[Proof of Theorem~\ref{thm:exp-conv}]
	It suffices to prove the theorem for $n =0$.

	Let us apply Proposition~\ref{prop:localize} with $\tilde{\rho}(\omega) = \rho(\omega)$ from Proposition~\ref{prop:local-conv}.  Let $(y_n, \eta_n)$ be a minimzier of $K_{-N, M}^{\omega, b}\varphi(y_M)$, 
	and $(\tilde{y}, \tilde{\eta}_n)$ a minimizer for
	\[
		\psi^-(y_N, N) = \left( K_{-M, N}^{\omega, b}\psi^-(\cdot, -M)  \right)(y_{N})
	\]
	with $y_M = \tilde{y}_M$, and $M = M_0(\omega)$. According to Proposition~\ref{prop:local-conv}, there exists  $C(N, \omega) \in \R$ such that for all $\|y - x_0^\omega\| < \rho(\omega)$, 
	\begin{equation}
		\label{eq:app-const}
		\left| K_{-N,0}^{\omega,b}\varphi(y) - \psi^-_\omega(y, 0)  - C(N,\omega) \right |  \le D_0(\omega) e^{-\lambda N}. 
	\end{equation}
	Then 
	\begin{align*}
		& K_{-N, M}^{\omega, b} \varphi(y_M) =  K_{-N, 0}^{\omega, b}\varphi(y_M) + A_{0, M}^{\omega,b}(y_0, y_M)              \\
		& = (K_{-N, 0}^{\omega,b}\varphi(y_0) - \psi^-_\omega(y_0, 0)) + \psi^-_\omega(y_0, 0) + A_{0, M}^{\omega,b}(y_0, y_M) \\
		& \ge C(N, \omega) + \psi^-(y_M, M) - D_0(\omega) e^{- \lambda N}.                                                     
	\end{align*}

	On the other hand, 
	\begin{align*}
		& \psi^-(\tilde{y}_M, M) = \psi^-(\tilde{y}_0, 0) + A_{0, M}^{\omega, b}(\tilde{y}_0, \tilde{y}_M)                                                                      \\
		& = (\psi^-(\tilde{y}_0, 0) - K_{-N, 0}^{\omega,b}\varphi(\tilde{y})) + K_{-N, 0}^{\omega,b}(\tilde{y}_0, \tilde{y}_M) + A_{0, M}^{\omega, b}(\tilde{y}_0, \tilde{y}_M) \\
		& \le -C(N, \omega) + K_{-N, M}^{\omega,b}\varphi(\tilde{y}_M) + D_0(\omega) e^{- \lambda N}.                                                                           
	\end{align*}
	Using $y_M = \tilde{y}_M$, combine both estimates, and take supremum over all $y_M$, we get 
	\[
		\sup_{y \in \T^d}\left|  K_{-N, M}^{\omega,b}\varphi(y) - \psi^-(y, M) - C(N, \omega)  \right| \le D_0(\omega) e^{- \lambda N}. 
	\]
	In order to shift the end time to $t=0$, let denote $E_Q = \{ \omega \in \Omega: \quad M_0(\omega) \le Q\}$ for $Q \in \N$. Fix some $Q$ such that  $P(E_Q) > 0$, then by ergodicity, almost surely, $M(\theta^k \omega) \le  Q$ for infinitely many $k$. Let us define $k(\omega)$ as the largest $k < -Q$ such that $M_0(\theta^k \omega) = Q$, then 
	\[
		\begin{aligned}
			& \left\| K_{-N, 0}^{\omega, b} \varphi - \psi^-_\omega(\cdot, 0) \right\| = \left\| K_{-N - k, k}^{\theta^k \omega, b} \varphi - \psi^-_{\theta^k \omega}(\cdot, - k) \right\| \\
			& \le D_0(\theta^k \omega) e^{-\lambda(N-k)} = D_0(\theta^k \omega) e^{\lambda k(\omega)} e^{-\lambda N}                                                                        
		\end{aligned}
	\]
	provided $N \ge \max\{Q, N_0(\theta^k \omega)\} - k$. By reducing $\lambda$ and taking $N$ larger we can absorb the constant $D_0(\theta^k \omega) e^{\lambda k(\omega)}$. 
\end{proof}

\section{Local convergence: localization and upgrade}
\label{sec:upgrade}

In this section we prove Proposition~\ref{prop:local-conv} (local convergence) using Proposition~\ref{prop:dt-localize} and a consecutive upgrade scheme. 
It is useful to have the following definition for book keeping. 
\begin{definition}\label{def:APm}
Given $\delta>0$, $N \in \N$, let $(y_n, \eta_n)_{n=-N}^0$ be a minimizer for $\psi^N_\omega(\cdot, 0) = K^{\omega,b}_{-N, 0}\varphi(\cdot)$, we say the orbit satisfies the (backward) $(\varphi, \delta, N)$ approximation property if for every $-N/3 \le n \le 0$ such that  $\|y_n - x_n^\omega\| < r(\theta^n \omega)$, we have
\[
	\left|
	\left( \psi^N_\omega(y_n, n) - \psi^N_\omega(x_n^\omega, n) \right) - \left( \psi^-_\omega(y_n, n) - \psi^-_\omega(x_n, n) \right)
	\right| < \delta. 	
\]
We denote this condition $\Apm(\omega, \varphi, \delta, N)$. 
\end{definition}
The following proposition is our main technical result, which says the approximation property allows us to estimate how close a backward minimizer is to the global minimizer:
\begin{proposition}\label{prop:dt-localize}
Let $0 < \epsilon < \lambda'/6$, there exists random variables $0<\rho(\omega)< r(\omega)$, $N_1(\omega)>0$ depending on $\epsilon$,  such that if a $\psi^N_\omega$ backward minimizer $(y_n, \eta_n)_{n = -N}^0$ satisfies the $\Apm(\omega, \varphi, \delta, N)$ condition, with \footnote{The exponential on lower bound in $\delta$ in \eqref{eq:ap-conditions} is for simpler calculation and comes with no loss of generality. Indeed, if $\delta$ is exponentially small we already have the conclusion of Proposition~\ref{prop:local-conv}.}
\begin{equation}
	\label{eq:ap-conditions}
	\|y_0 - x_0^\omega\| < \rho(\omega), \quad 
	N > N_1(\omega), \quad \delta^{\frac18} < \rho(\omega), \quad  \delta \ge e^{-(\lambda'-3\epsilon)N/3},
\end{equation}
then there exists $\max\left\{  \frac{1}{8\epsilon} \log \delta, - \frac{N}{6}  \right\} \le k <0$ such that 
\begin{equation}
	\label{eq:yk-delta}
	\|y_k - x_k^\omega\| < \max\left\{ \delta^q, \quad e^{-(\lambda'- 3\epsilon)N/6} \right\},
\end{equation}
where $q = (\lambda'-3\epsilon)/(8\epsilon)$. 
\end{proposition}
The proof require a detailed analysis using hyperbolic theory, and is deferred to the next few sections. In this section we prove Proposition~\ref{prop:local-conv} assuming Proposition~\ref{prop:dt-localize}.

We need to use both the forward and backward dynamics.

\begin{definition}\label{def:APp}
Given $\delta>0$, $N \in \N$, let $(y_n, \eta_n)_{n=0}^N$ be a minimizer for $\chpsi^N_\omega(\cdot, 0) = \chK^{\omega,b}_{0, N}\varphi(\cdot)$, we say the orbit satisfies the (forward) $(\varphi, \delta, N)$ approximation property, if for every $0 \le k \le N/3$ such that  $\|y_k - x_k^\omega\| < r(\theta^k \omega)$, we have
\[
	\left|
	\left( \chpsi^N_\omega(y_k, k) - \chpsi^N_\omega(x_k^\omega, k) \right) - \left( \psi^+_\omega(y_k, k) - \psi^+_\omega(x_k, k) \right)
	\right| < \delta. 
\]
We  denote this condition $\App(\omega, \varphi, \delta, N)$. 
\end{definition}

We state a forward version of Proposition~\ref{prop:dt-localize}. The proof is the same. 
\begin{proposition}\label{prop:forward- dt-localize}
There exists  random variables $0<\check{\rho}(\omega)< r(\omega)$ and $\check{N}_1(\omega)>0$ such that if a $\check{\psi}^N_\omega$ backward minimizer $(y_n, \eta_n)_{n = 0}^N$ satisfies the $\App(\omega, \varphi, \delta, N)$ condition, and in addition, 
\begin{equation}
	\label{eq:for-ap-conditions}
	\|y_0 - x_0^\omega\| < \check{\rho}(\omega),  \quad 
	N > \check{N}_1(\omega),  \quad \delta^{\frac18} < \check{\rho}(\omega),\quad \delta \ge \check{C}_2(\omega)e^{-(\lambda'-2\epsilon)N/3},
\end{equation}
then there exists $0 < k \le \max\left\{ - \frac{1}{8\epsilon} \log \delta,  \frac{N}{6}  \right\}$ such that 
\begin{equation}
	\label{eq:yk-delta}
	\|y_k - x_k^\omega\| < \max\left\{ \delta^q, \quad e^{-(\lambda'- 3\epsilon)N/6} \right\},
\end{equation}
where $q = (\lambda'-3\epsilon)/(8\epsilon)$. 
\end{proposition}


The main idea for the proof of Proposition~\ref{prop:local-conv} is to use both the forward and backward dynamics to repeatedly upgrade the estimates. If we have a $\Apm$ condition, Proposition~\ref{prop:dt-localize} implies upgraded localization of \emph{backward} minimizers at a earlier time. This can be applied to get a better approximation of the \emph{forward} solution for the later time, obtaining an improved $\App$ condition. We then reverse time and repeat. However, due to technical reasons, we can only apply this process on a sub-interval called a good interval. 

\begin{definition}
	For $\beta>0$, we say $j \in [-N, 0]$ is a backward $\beta$-good time if for every $\varphi \in C(\T^d)$ and every $K_{-N,0}^{\omega, b}\varphi$ minimizer $(y_n, \eta_n)_{n = -N}^0$, we have 
	\[
		\|y_j - x_j^\omega\| < \beta < \rho(\theta^j\omega), \quad K^\epsilon(\theta^j \omega) < \beta^{-1}, \quad N_1(\omega) < \beta^{-1}
	\]
	where $\rho(\omega)$ is in Proposition~\ref{prop:dt-localize}. Define forward $\beta$-good time similarly, by using forward minimizers and $\check{\rho}$ from Proposition~\ref{prop:forward- dt-localize}. An interval $[n_1, n_2]$ is good if $n_2$ is backward good and $n_1$ is forward good. 
\end{definition}

\begin{figure}[t]
	\centering
	\includegraphics[width=2.5in]{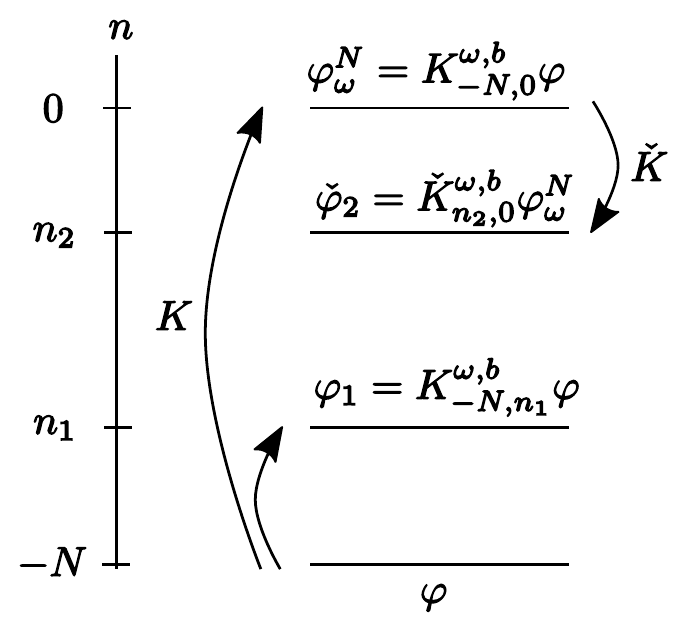}
	\caption{Notations for forward-backward iteration}\label{fig:reg-int}
\end{figure}

Write $\varphi^N_\omega = K_{-N, 0}^{\omega, b}\varphi$, note that by Corollary~\ref{cor:two-way-minimizer}, if $(y_n, \eta_n)$ is a minimizer for $K_{-N, 0}^{\omega, b}\varphi$, then it is also a minimizer for $\chK_{-N, 0}^{\omega, b}\varphi^N_\omega = \chK_{-N, 0}^{\omega, b} K_{-N, 0}^{\omega, b}\varphi$. For a $\beta-$good interval, denote
\begin{equation}
	\label{eq:def-interval}
	\begin{aligned}
		& \omega_1 = \theta^{n_1}\omega, \quad \varphi_1 = K_{-N, n_1}^{\omega, b}\varphi,                                 \\
		& \omega_2 = \theta^{n_2}\omega, \quad \chphi_2 = \chK_{n_2, 0} \varphi^N_\omega = \chK_{n_2, 0} K_{-N, 0}\varphi. 
	\end{aligned}
\end{equation}
See Figure~\ref{fig:reg-int} for a visualization of the notations.

We now describe the upgrade lemma:
\begin{lemma}\label{lem:backforw}
For $0 < \epsilon < \lambda' /12$ and $\beta>0$, there exists $N_2(\omega)>0$ such that if 
\[
	N>N_2(\omega), \quad \delta^{\frac18} < \beta/2, \quad 
	\delta \ge e^{-(\lambda' - 3\epsilon)N/3},
\]
and if $[-5N/9, -4N/9] \subset [n_1, n_2] \subset [-2N/3, -N/3]$  is a $\beta$-good time interval with regards to $[-N,0]$, we have:
\begin{enumerate}
	\item For $\omega_1, \omega_2, \varphi_1, \chphi_2$ defined in \eqref{eq:def-interval} and  $\barN = n_2 - n_1$,  if 
	\[
		(y_{j+n_2}, \eta_{j +n_2})_{j = - \barN}^0  \text{ satisfies } 
		\Apm(\omega_2, \varphi_1, \delta, \barN)
		\text{ condition,}
	\]
	then the orbit 
	\[
		(y_{j+n_1}, \eta_{j +n_1})_{j = 0}^\barN  \text{ satisfies } 
		\App(\omega_1, \chphi_2, \delta_1, \barN)
		\text{ condition,}
	\]
	where 
	\begin{equation}
		\label{eq:delta1}
		\delta_1 = \max\left\{\delta^{q - \frac14}, e^{-(\lambda' - 6\epsilon)N/54} \right\}.  
	\end{equation}
	\item If
	\[
		(y_{j+n_1}, \eta_{j +n_1})_{j = 0}^\barN  \text{ satisfies } 
		\App(\omega_1, \chphi_2, \delta, \barN)
		\text{ condition,}
	\]
	then 
	\[
		(y_{j+n_2}, \eta_{j +n_2})_{j = - \barN}^0  \text{ satisfies } 
		\Apm(\omega_2, \varphi_1, \delta_1, \barN)
		\text{ condition,}
	\]
	with $\delta_1$ given by \eqref{eq:delta1}. 
\end{enumerate}
\end{lemma}

The reason to require the good time interval to lie in $[-2N/3, -N/3]$ is to apply the following lemma, which says the global minimizer $x_n^\omega$ is almost a minimizer for finite time solution in the middle of the time interval: 
\begin{lemma}\label{lem:middle-almost-min}
There exists $N_3(\omega)>0$ such that if $N > N_3(\omega)$, the following holds almost surely, for arbitrary $\varphi \in C(\T^d)$ and $-2N/3 \le j < k \le -N/3$: 
\[
	\left| K_{-N, j}^{\omega, b} \varphi(x_j^\omega)  - K_{-N, k}^{\omega, b} \varphi (x_k^\omega) + A_{j, k}^{\omega, b}(x_j^\omega, x_k^\omega) \right| < e^{-(\lambda' - 3\epsilon)N/6}, 
\]
\[
	\left| \chK_{j, 0}^{\omega, b} \varphi(x_j^\omega)  - \chK_{k, 0}^{\omega, b} \varphi (x_k^\omega) + A_{j, k}^{\omega, b}(x_j^\omega, x_k^\omega) \right| < e^{-(\lambda' - 3\epsilon)N/6}. 
\]	
\end{lemma}
The proof of the lemma is deferred to section~\ref{sec:guiding}. 

\begin{figure}[t]
	\centering
	\includegraphics[width = 2.5in]{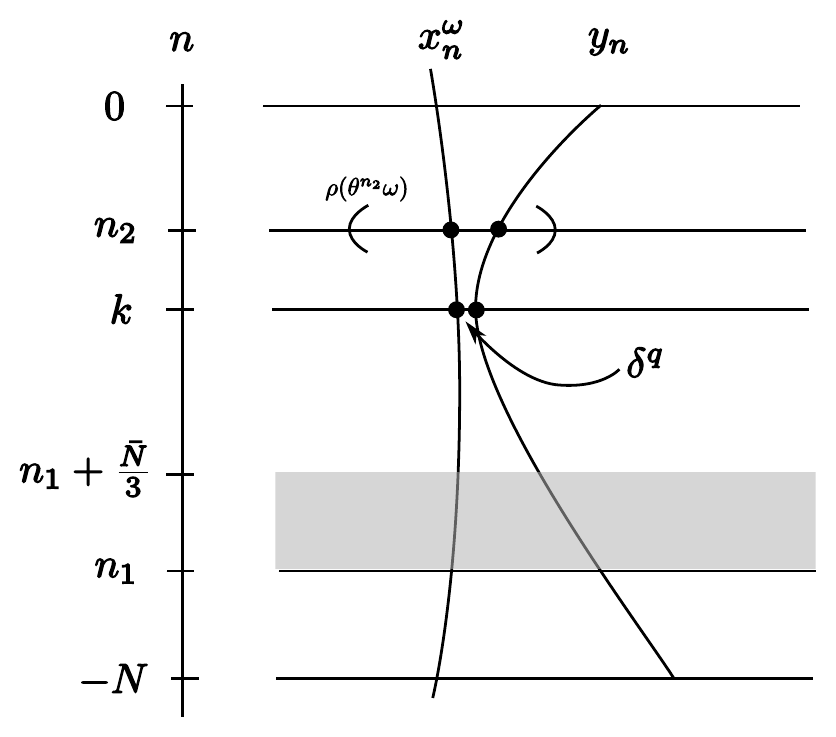}
	\caption{Proof of Lemma~\ref{lem:backforw}, part (1): Use early localization to estimate forward solutions on the gray interval}\label{fig:upgrade}
\end{figure}

\begin{proof}[Proof of Lemma~\ref{lem:backforw}]
	Refer to Figure~\ref{fig:upgrade} for an illustration of the strategy of the proof. 

	Since $n_2$ is a backward good time and $\delta^{\frac18} < \beta < \rho(\theta^{n_2}\omega) = \rho(\omega_2)$, condition \eqref{eq:ap-conditions} holds for the orbit  $(y_{j+n_2}, \eta_{j +n_2})_{j = - \barN}^0$ at the shifted time $\omega_2$, the initial condition $\varphi_1$, and interval size $\barN$.

	Therefore by Proposition~\ref{prop:dt-localize}, there exists $0 >  - n_2 \ge \max\{-\frac{\barN}6, \frac{1}{8\epsilon} \log \delta\}$, such that  
	\[
		\|y_k - x_k^\omega\| < \max\{\delta^q, e^{-(\lambda' - 3\epsilon)\barN/6}\}. 
	\]
	Also, $\barN \ge N/9 \ge \beta^{-1} \ge N_1(\omega_2)$ provided $N \ge 9 \beta^{-1}$. 
	Consider $n_1 \le j \le \barN/3 + n_1 < k$, such that $\|y_j - x_j^\omega\|<r(\theta^j \omega)$. By Corollary~\ref{cor:two-way-minimizer},
	\[
		\chK_{j, 0}\varphi^N_\omega(y_j) = \chK_{k, 0}\varphi^N_\omega(y_k) - A_{k, j}(y_k, y_j),
	\]
	then
	\begin{equation}
		\label{eq:yk-first-est}
		\begin{aligned}
			& \left|                                                                                                                                                                          
			\chK_{j, 0}\varphi^N_\omega(y_j) - \chK_{k, 0}\varphi^N_\omega(x_k^\omega) + A_{k,j}(x_k^\omega, y_j)
			\right|	 \\
			& \le \left| \chK_{k, 0}\varphi^N_\omega(y_k) - \chK_{k, 0}\varphi^N_\omega(x_k^\omega)  \right| +                                                                                
			\left| A_{k,j}(y_k, y_j) - A_{k,j}(x_k^\omega, y_j) \right| \\
			& \le K^\epsilon(\theta^k \omega) \|x_k^\omega - y_k\|                                                                                                                            
			\le K^\epsilon(\theta^{n_2}\omega) e^{\epsilon|k-n_2|} \max\left\{   \delta^q, \quad e^{-(\lambda' - 3\epsilon)\barN/6}  \right\} \\
			& \le \beta^{-1} \max\left\{   \delta^{q-\frac18}, \quad e^{-(\lambda' - 4\epsilon)\barN/6}  \right\} \le \beta^{-1} \max\{\delta^{q-\frac18}, e^{-(\lambda' -4\epsilon)N/54}\} . 
		\end{aligned}
	\end{equation}
	using Lemma~\ref{lem:lipsht-est}, and  the estimates  $e^{\epsilon |k-n_2|} \le \delta^{-\frac18}$,  $e^{\epsilon|k-n_2|} \le e^{\epsilon \barN/6}$ and $\barN \ge N/9$. 

	On the other hand, the dual version of Lemma~\ref{lem:phi-minus-approx} implies 
	\[
		\begin{aligned}
			& \left| 
			\psi^+_\omega(y_j, j) - \psi^+_\omega(x_k^\omega, k) + A_{k, j}(x_k^\omega, x_j^\omega) 
			\right| < C_1^\epsilon(\theta^j\omega) e^{-(\lambda' - \epsilon)(k-j)} \\
			& \le   C_1^\epsilon(\omega) e^{\epsilon |j|} e^{-(\lambda' - \epsilon)\barN/2} \le C_1^\epsilon(\omega) e^{2\epsilon N/3} e^{-(\lambda'-\epsilon)N/18} = C_1^\epsilon(\omega) e^{-(\lambda' - 13\epsilon)N/18}. 
		\end{aligned}
	\]
	Combine the estimates, we get 
	\[
		\begin{aligned}
			& \left|                                                       \left( \chK_{j, 0}^{\omega, b} \varphi^N_\omega (y_j) - \chK_{k, 0}^{\omega, b} \varphi^N_\omega (x_k^\omega) \right) - 
			\left( \psi^+_\omega(y_j, j) - \psi^+_\omega(x_k^\omega, k) \right)
			\right| \\
			& \le  \beta^{-1} \max\left\{   \delta^{q-\frac18}, \quad e^{-(\lambda' - 4\epsilon)N/54}  \right\} + C_1^\epsilon(\omega) e^{-(\lambda' - 13\epsilon)N/18}.                                       
		\end{aligned}\]
		We now apply Lemma~\ref{lem:middle-almost-min} to $\varphi = \varphi^N_\omega$, to replace the index $k$ with $j$:
		\begin{equation}
			\label{eq:unshifted}
			\begin{aligned}
				& \left|                                                                                                                                                                         
				\left( \chK_{j, 0}^{\omega, b} \varphi^N_\omega (y_j) - \chK_{j, 0}^{\omega, b} \varphi^N_\omega (x_j^\omega) \right) -
				\left( \psi^+_\omega(y_j, j) - \psi^+_\omega(x_j^\omega, j) \right)
				\right| \\
				& \le \beta^{-1} \max\left\{   \delta^{q-\frac18}, \quad e^{-(\lambda' - 4\epsilon)N/54}  \right\} + C_1(\omega) e^{-(\lambda' - \epsilon)N/3} + e^{-(\lambda' - 3\epsilon) N/6} \\
				& \le 2 \beta^{-1}\max\left\{\delta^{q - \frac18}, e^{-(\lambda' - 4\epsilon)N/54} \right\} < \max\{\delta^{q - \frac14}, e^{-(\lambda' - 5\epsilon)N/54}\} = \delta_1,          
			\end{aligned}
		\end{equation}
		where in the last inequality, we take $N_2(\omega)$ large enough so that 
		$C_1(\omega) e^{-(\lambda' - \epsilon)N/3} + e^{-(\lambda' - 3\epsilon) N/6} < e^{-(\lambda' - 4\epsilon)N/54}$ and $2\beta^{-1} < e^{-\epsilon N/54}$, then we use $2\beta^{-1} < \delta^{-\frac18}$. 

		Observe that by the standard semi-group property, 
		\[
			\chK_{j - n_1, \barN}^{\omega_1, b} \chphi_2 = \chK_{j-n_1, \barN}^{\theta^{n_1}\omega, b} \, \chK_{n_2, 0}^{\omega, b}\, \varphi_\omega^N 
			= \chK_{j, n_2}^{\omega, b} \chK_{n_2, 0}^{\omega, b} \varphi^N_\omega 
			= \chK_{j, 0}^{\omega, b}\varphi^N_\omega, 
		\]
		substitute into \eqref{eq:unshifted} we obtain
		$\App(\omega_1, \chphi_2, \delta_1, \barN)$. 

		We now discuss case 2. Starting with the condition $\App(\omega_1, \chphi_2, \delta, \barN)$, we obtain $0 \le k_* - n_1 \le  \max\left\{ - \frac{1}{8\epsilon} \log \delta,  \frac{N}{6}  \right\}$ such that 
		\[
			\|y_{k_*} - x_{k_*}\| \le \max\left\{ \delta^q, \quad e^{-(\lambda'- 3\epsilon)N/6} \right\}. 
		\]
		Then for all $k_* < n_2 - \barN/3 \le j \le n_2$, if $\|y_j - x_j^\omega\|< r(\theta^j \omega)$, using the fact that $(y_n, \eta_n)$ is a minimizer for $K_{-N, 0}\varphi$, similar to \eqref{eq:yk-first-est} we have
		\[
			\left| K_{-N, j}^{\omega, b}\varphi(y_j) - K_{-N, k}^{\omega, b}\varphi(x_k^\omega) + A_{j, k}^{\omega, b}(x_k^\omega, y_j) \right| \le \beta^{-1}\max\left\{ \delta^{q - \frac18}, e^{-(\lambda' - 4\epsilon)N/54} \right\}, 
		\]
		and following the same strategy as before, we get
		\[
			\left| \left( K_{-N, j}^{\omega, b}\varphi(y_j) - K_{-N, j}^{\omega, b}\varphi(x_j^\omega) \right) - (\psi^-(y_j, j) - \psi^-(x_j^\omega, j)) \right| < \delta_1.  
		\]
		Since 
		\[
			K_{-\barN, j - n_2}^{\omega_2, b}\varphi_1 = K_{n_1, j}^{\omega, b}	K_{-N, n_1}\varphi = K_{-N, j}^{\omega, b} \varphi, 
		\]
		we obtain $\Apm(\omega_2, \varphi_1, \delta_1, \barN)$. 
	\end{proof}

	To carry out the upgrading procedure, we need to show that $\beta-$good time intervals exist. 

	\begin{lemma}\label{lem:regular-interval}
	There exists $\beta_0>0$ such that, for any $0 < \beta \le \beta_0$, there exists $N_3(\omega)>0$, and for all $N> N(\omega)$ there exists a $\beta-$regular time interval $[-5N/9, -4N/9] \subset [n_1, n_2] \subset [-2N/3, -N/3]$. 
\end{lemma}
\begin{proof}
	Let $\beta_0>0$ be small enough that 
	\[
		P(\rho(\omega) > \beta_0, \, K^\epsilon(\omega) < \beta_0^{-1}, 
		\, N_1(\omega)< \beta_0^{-1}) > \frac{17}{18}. 
	\]
	By Proposition~\ref{prop:localize}, for any $0 < \beta \le \beta_0$, there exists $M_0(\omega)>0$ such that any minimizer $(y_n, \eta_n)_{n=-M}^N$ with $M, N > M_0(\omega)$ satisfies $\|y_0 - x_0^\omega\|< \beta$. We now choose $\beta_1 >0$ small enough such that 
	\[
		P(\rho(\omega) > \beta, \, K^\epsilon(\omega) < \beta^{-1}, 
		\, N_1(\omega)< \beta^{-1}, \, M_0(\omega) < \beta^{-1}) > \frac{8}{9}. 
	\]
	Then, there exists $N_3(\omega)>0$ such that for all $N> N_3(\omega)$ the density of $\beta-$regular $n$ in $[-N, 0]$ is larger than $\frac89$. In particular, the interval $[-4N/9, -N/3]$  must contain a regular time $n_2$. We impose $N_3(\omega) > 3 \beta_1^{-1}$, then Proposition~\ref{prop:localize} implies for any $N > N_3(\omega)$, $\|y_{n_2}- x_{n_2}^\omega\| <  \beta \le \rho(\omega)$, therefore $n_2$ is a good time. 

	Apply the same argument, by possibly choosing a different $N_3(\omega)$, we can find a forward good time $n_1$ in $[-2N/3, -5N/9]$.
\end{proof}

\begin{proof}[Proof of Proposition~\ref{prop:local-conv}]
	By Lemma~\ref{lem:regular-interval} there exists a $\beta-$good interval. 
	We first show that if for $K_{-N, 0}^{\omega, b}\varphi$ backward minimizer $(y_n, \eta_n)_{n = -N}^0$, the condition 
	\begin{equation}
		\label{eq:exp-apm}
		\Apm(\omega_2, \varphi_1, e^{- \lambda_1 N}, \bar{N})
	\end{equation}
	holds for an explicitly defined $\lambda_1$, then Proposition~\ref{prop:local-conv} follows. Indeed, we only need the estimate 
	\[
		\left|
		\psi^N_\omega(y_{n_2}, n_2) - \psi^-_\omega(y_{n_2}, n_2)
		- C(n_2, \omega, \varphi)
		\right| \le  e^{- \lambda_1 N},  
	\]
	where $C(n_2, \omega, \varphi) = \psi^N_\omega(x_{n_2}^\omega, n_2) - \psi^-_\omega(x_{n_2}^\omega, n_2)$. 

	On one hand, 
	\[
		\psi^N_\omega(y_0, 0) - \psi^-(y_0, 0) \ge \psi^N(y_{n_2}, n_2) - \psi^-(y_{n_2}, n_2) \ge C(n_2, \omega, \varphi) - e^{- \lambda_1 N},
	\]
	on the other hand, by Lemma~\ref{lem:phi-minus-approx}, 
	\[
		\begin{aligned}
			& \psi^N_\omega(y_0, 0) - \psi^-(y_0, 0) \le \psi^N(x_{n_2}^\omega, n_2) - \psi^-(x_{n_2}^\omega, n_2) + C_1(\omega) e^{-(\lambda' -\epsilon)|n_2|} \\
			& =	C(n_2, \omega, \varphi) + C_1(\omega) e^{-(\lambda' - \epsilon)N/3}                                                                             
			\le C(n_2, \omega, \varphi) + e^{-(\lambda' - 2\epsilon)N/3} 
		\end{aligned}
	\]
	if $N$ is large enough. Proposition~\ref{prop:local-conv} follows by taking $\lambda = \min\{\lambda_1, (\lambda' - 2\epsilon)/3\}$.

	We now prove \eqref{eq:exp-apm}. 
	Choose $\beta = \beta_0$ as in Lemma~\ref{lem:regular-interval}, and $\delta$ such that 
	\begin{equation}
		\label{eq:dt-initial}
		\delta^{\frac18} < \min\left\{  \rho(\omega), \beta_0/2\right\}.   
	\end{equation}
	Using Proposition~\ref{prop:backward-min}, there is $N_0(\omega)$ large enough such that for all $N_0 > N(\omega)$, all $\varphi \in C(\T^d)$, we have 
	\[
		\|K_{-N, n}^{\omega, b}\varphi - \psi^-_\omega(\cdot, n)\|_* < \delta, \quad -2N/3 \le n \le 0.
	\]
	In particular,  for any minimizer $(y_n, \eta_n)$, we have
	\[
		(y_{j+n_2}, \eta_{j +n_2})_{j = - \barN}^0  \text{ satisfies } 
		\Apm(\omega_2, \varphi_1, \delta, \barN)
		\text{ condition.}
	\]
	Apply Lemma~\ref{lem:backforw}, obtain 
	\[
		(y_{j+n_1}, \eta_{j +n_1})_{j = 0}^\barN  \text{ satisfies } 
		\App(\omega_1, \chphi_2, \delta_1, \barN)
		\text{ condition,}
	\]
   with $\delta_1 = \max\{\delta^{q - \frac14}, e^{-(\lambda' - 5\epsilon)N/54} \}$. 

	Now we are going to apply Lemma~\ref{lem:backforw} repeatedly, from $\Apm$ to $\App$ and back until a desired estimate for $\delta$ is achieved. We shall assume thta $N > N_2(\omega)$. On the first step we get an estimate $\App(\omega_1, \varphi_2, \delta_1, \barN)$ for $(y_{j+n_2}, \eta_{j+n_2})_{j = - \barN}^0$ where $\delta_1 = \max\{\delta^{q - \frac14}, e^{-(\lambda' - 5\epsilon)N/54}\}$. Since $q - \frac14 >1$ this estimate is an improvement of $\delta$ unless $\delta < e^{-(\lambda' - 5\epsilon)N/54}$. Notice that if this happens we have already proven our statement with $\lambda_1 = (\lambda' - 5\epsilon)/54$. It is easy to see that the level $\delta  < e^{-(\lambda' - 5\epsilon)N/54}$ will be reached in a finite number of steps depending on $N$. Notice $N$ is large enough but fixed, this finishes the proof.  
\end{proof}

\section{Properties of the finite time solutions}
\label{sec:finite-time}

We have proven all our statements except Proposition~\ref{prop:dt-localize}, which we prove in the next two sections. 
\subsection{The guiding orbit}\label{sec:guiding}

For $N \in \N$,  denote 
\[
	\psi^N_\omega(x, n) = K_{-N, n}^{\omega,b}\varphi(x), \quad -N \le n \le 0. 
\]
We define
\begin{equation}
	\label{eq:def-QN}
	Q^N_\omega(x,n) = \psi^N_\omega(x,n) - \psi^+_\omega(x, n), \quad - N \le n \le 0,
\end{equation}
which is a finite time analog of $Q^\infty_\omega$. (Again, the subscript $\omega$ may be dropped). 

The function $Q^N$ is a Lyapunov function for minimizers, in the following sense:
\begin{lemma}\label{lem:QN-lya}
Let $(y_n, \eta_n)_{n =- N}^0$ be a minimizer for $K_{-N, 0}^{\omega, b}\varphi(y_0)$ (will use $\psi^N(x,0)$ from now on). Then for all $-N \le j < k \le 0$, 
\[
	Q^N(y_j, j) \le Q^N(y_k, k). 
\]
\end{lemma}
\begin{proof}
	By definition, 
	\[
		\begin{aligned}
			& Q^N(y_k, k) = \psi^N(y_k, k) - \psi^+(y_k, k) =                                   
			\psi^N(y_j, j) + A_{j, k}^{\omega, b}(y_j, y_k) - \psi^+(y_k, k) \\
			& \ge \psi^N(y_j, j) - \left( \chK_{j,k}^{\omega, b}\psi^+(\cdot, k) \right)(y_j) = 
			\psi^N(y_j, j) - \psi^+(y_j, j) = Q^N(y_j, j). 
		\end{aligned}	
	\]
\end{proof}

Let
\[
	z_{-N}^\omega \in \argmin_z Q^N(z, -N),
\]
and define $(z_n^\omega, \zeta_n^\omega)_{n = -N}^\infty$ to be a forward minimizer starting from $z_{-N}^\omega$. The orbit $(z_n^\omega, \zeta_n^\omega)$ plays the role of the global minimizer $(x_n^\omega, v_n^\omega)$ in the finite time set up, and is called the \emph{guiding orbit}. The choice of $z_N^\omega$ may not be unique, but our analysis will not depend on the choice of $z_N^\omega$. 

\begin{lemma}\label{lem:prop-fin-min} The guiding orbit has the following properties. 
\begin{enumerate}
	\item 
	\[
		z_n^\omega \in \argmin_z Q^N(z, n), \quad -N \le n \le 0.
	\]
	\item 
	\[
		\zeta_n^\omega = \nabla \psi^N(z_n^\omega, n) +b  =  \nabla \psi^+(z_n^\omega, n) + b, \quad - N \le n \le 0.
	\]
	where both gradients exists. 
	\item 
	\[
		Q^N(z_j^\omega, j) = Q^N(z_k^\omega, k), \quad -N \le j < k \le 0. 
	\]
	\item $z_k^\omega$, $-N \le k \le 0$ is a backward minimizer for $K_{-N, 0}^{\omega,b} \varphi$. 
\end{enumerate}
\end{lemma}
\begin{proof}
	We prove (3) first. Since $(z_j^\omega)_{j \ge -N}$ is a forward minimizer, we have 
	\begin{equation}
		\label{eq:QN-eq}
		\begin{aligned}
			& Q^N(z_j^\omega, j) = \psi^N(z_j^\omega, j) - \psi^+(z_j^\omega, j) 
			= \psi^N(z_j^\omega, j) - \psi^+(z_k^\omega, k) + A_{j,k}^{\omega, b}(z_j^\omega, z_k^\omega) \\
			& \ge \psi^N(z_k, k) - \psi^+(z_k^\omega, k) = Q^N(z_k^\omega, k).   
		\end{aligned}
	\end{equation}
	On other other hand, let $y_k \in \argmin Q^N(\cdot, k)$, and let $(y_n)_{n=-N}^k$ be a minimizer for $\varphi$ ending at $y_k$. Then by an argument similar to Lemma~\ref{lem:QN-lya}, for any $y \in \T^d$,
	\begin{equation}
	  \label{eq:QNy-zj}
	  Q^N(y, k) \ge Q^N(y_k, k) \ge Q^N(y_{-N}, -N) \ge Q^N(z_{-N}^\omega, -N). 
	\end{equation}
	In particular, taking $y = z_k^\omega$, we have $Q^N(z_k^\omega, k) \ge Q^N(z_{-N}^\omega, -N)$. Using \eqref{eq:QN-eq} for $j = -N$, we get $Q^N(z_{-N}^\omega, -N) = Q^N(z_k^\omega, k)$ which implies (3). 

	This also implies that \eqref{eq:QN-eq} is in fact an equality, therefore $\psi^N(z_j^\omega, j) = \psi^N(z_k^\omega) - A_{j,k}^{\omega, b}(z_j^\omega, z_k^\omega)$ and (4) follows. 

	Using again \eqref{eq:QNy-zj}, we have 
	\[
		\min_y Q^N(y, k) \ge Q^N(z_{-N}^\omega, -N) = Q^N(z_k^\omega, k)
	\]
	which implies (1). Finally, since $\psi^N(\cdot, n)$ is a semi-concave function for $n \ge -N$ and $\psi^+$ is semi-convex, (2) follows from Lemma~\ref{lem:grad-lip}. 
\end{proof}

Combine (3) of Lemma~\ref{lem:prop-fin-min} with Lemma~\ref{lem:QN-lya},  we get 
\begin{equation}
	\label{eq:lyapunov}
	Q^N(y_j, j) - Q^N(z_j^\omega, j)  \le Q^N(y_k, k) - Q^N(z_k^\omega, k).
\end{equation}

\subsection{Regular time and localization of the guiding orbit}
\label{sec:local-guiding}

We use the idea of regular time again. Let   $0 < \beta <1$ be such that 
\[
	P(C(\omega) < \beta^{-1}, r(\omega) > \beta) > \frac{23}{24},
\]
with $C(\omega), r(\omega)$ from Proposition~\ref{prop:unst-stab}. Let $M_0(\omega)$ be the random variable given by Proposition~\ref{prop:localize} with $\tilde{\rho} = \beta$. Let $\beta_1 >0$ be such that $P(M_0(\omega)>\beta_1^{-1})< 1/24$.  $n$ is called regular if 
\[
	C(\theta^n \omega) < \beta^{-1}, \quad r(\theta^n \omega) > \beta,
	\quad M_0(\omega) < \beta_1^{-1},
\]
using same proof as Lemma~\ref{lem:regular-interval}, we get 
\begin{lemma}\label{lem:regular}
There exists $N_1(\omega)>0$ such that for all $N > N_1(\omega)$, there exists a regular time in each time interval of size at least $N/12$ contained in $[-N, 0]$. 
\end{lemma}

\begin{lemma}\label{lem:approx-guiding}
There exists $N_4(\omega)>0$ depending on $\beta$, $\epsilon$ such that for all  $N > N_4(\omega)$, and  $- 5N/6 \le  k \le 0$, we have 
\[
	\| z_k^\omega - x_k^\omega\| \le \beta^{-1} e^{-\lambda'|k + 5N/6|}.
\]
By possibly enlarging $N_4(\omega)$, we have 
\[
	0 \le \psi^N(x_j^\omega) + A_{j, k}^{\omega, b}(x_j^\omega, x_k^\omega) - \psi^N(x_k^\omega) \le e^{-(\lambda' - 3\epsilon)N/6}, 
	\quad -2N/3 \le j < k \le 0. 
\]
\end{lemma}
\begin{proof}
	The proof is again very similar to that of Lemma~\ref{lem:regular-interval}. Let $n_*$ be a regular time in $[-5N/6, - 11 N/ 12]$. Then 
	\[
		\|z_{n_*}^\omega - x_{n_*}^\omega\| < \beta < r(\theta^{n_*}\omega). 
	\]
	Apply Proposition~\ref{prop:hyp}, we get 
	\[
		\|z_k^\omega - x_k^\omega\| \le C(\theta^{n_*} \omega) e^{-\lambda'|k-n_*|} \le \beta^{-1} e^{-\lambda'|k-n_*|}. 
	\]
	Since $n_* \ge -5N/6$, the first estimate follows. For the second estimate, 
	since $z_k^\omega$ is a minimizer for $K_{-N, 0}\varphi$, we have 
	\[
		0 =  \psi^N(z_j^\omega) + A_{j, k}^{\omega, b}(z_j^\omega, z_k^\omega) - \psi^N(z_k^\omega), \quad -2N/3 \le j < k \le 0. 
	\]
	To avoid magnifying the coefficient of $\epsilon$, let $K^{\epsilon/2} > K(\omega)$ be the result of applying Lemma~\ref{lem:tempered-arK} with parameter $\epsilon/2$, then 
	\[
		\begin{aligned}
			& \psi^N(x_j^\omega) + A_{j, k}^{\omega, b}(x_j^\omega, x_k^\omega) - \psi^N(x_k^\omega) \le   2K^{\epsilon/2}(\theta^j \omega) \|z_j^\omega - x_j^\omega\|  + 2K^{\epsilon/2}(\theta^k \omega)  \|z_k^\omega - x_j^\omega\| \\
			& \le 4K^{\epsilon/2}(\omega) e^{\frac{\epsilon}2 \cdot \frac{2N}3} e^{-\lambda' N/6} \le 4 K^{\epsilon/2}(\omega) e^{-(\lambda' - 2\epsilon)N/6} \le e^{-(\lambda' - 3\epsilon)N/6},                                        
		\end{aligned}
	\]
	where the last step is achieved by taking $N_4(\omega)$ large enough. 
\end{proof}

We now prove Lemma~\ref{lem:middle-almost-min}. 
\begin{proof}[Proof of Lemma~\ref{lem:middle-almost-min}]
	We note that Lemma~\ref{lem:approx-guiding} proves half of the estimates in Lemma~\ref{lem:middle-almost-min}. The other half is proven using the same argument and reversing time. 
\end{proof}

For the rest of the paper, we will only deal with time $n \ge -N/3$. We record the improved estimates on this time interval. 
\begin{lemma}\label{lem:xk-zk-comp}
There $N_4(\omega)>0$ such that for $N > N_4(\omega)$, the following holds for  $-N/3 \le k \le 0$:
let $f$ be any of the following functions: $\psi^N(\cdot, k)$, $\psi^\pm(\cdot,k)$, $A_{m,k}(y, \cdot)$ or $A_{k,n}(\cdot, y)$, . Then 
\[
	|f(x_k^\omega) - f(z_k^\omega)| \le e^{-(\lambda - 2\epsilon)N/3}. 
\]
\end{lemma}
\begin{proof}
	We note that all choices of $f$ are $K(\theta^k\omega)$ Lipshitz functions. Then 
	\begin{align*}
		& |f(x_k^\omega) - f(z_k^\omega)| \le K(\theta^k \omega) \|x_k^\omega - z_k^\omega\| \le K^\epsilon(\theta^k \omega) \beta^{-1} e^{-\lambda' N/3} \\
		& \le K^\epsilon(\omega) e^{\epsilon|k|} \beta^{-1} e^{-\lambda' N/3}                                                                             
		\le \beta^{-1}K^\epsilon(\omega) e^{-(\lambda' - \epsilon\epsilon)N/3} \le e^{-(\lambda' - 2\epsilon)N/3}
	\end{align*}
	for $N$ large enough. 
\end{proof}

\subsection{Stability of the finite time minimizers}

We show that if an orbit $(y_n, \eta_n)_{n = -N}^0$ satisfies $\Apm(\omega, \varphi, \delta, N)$ condition, then it is stable in the backward time. First, we obtain an analog of \eqref{eq:metric}. 
\begin{lemma}\label{lem:QN-metric}
Assume the orbit $(y_n, \eta_n)$ satisfies $\Apm(\omega, \varphi, \delta, N)$, 
Then for each $- N/3 \le k \le 0$ such that $\|y_k - x_k^\omega\| < r(\theta^k \omega)$, we have
\[
	Q^N(y_k, k) - Q^N(x_k^\omega, k) \ge a(\theta^k \omega)\|y_k - x_k^\omega\|^2 -  \delta,
\]
\[
	Q^N(y_k, k) - Q^N(x_k^\omega, k) \le K(\theta^k \omega) \|y_k - x_k^\omega\|^2 + \delta. 
\]
\end{lemma}
\begin{proof}
	The definition of $\Apm$ implies that for all $\|y_k - x_k^\omega\| < r(\theta^k \omega)$, 
	\[
		\|\left( Q^N(y_k, k) - Q^N(x_k^\omega, k) \right)
		- \left( Q^\infty(y_k, k) - Q^\infty(x_k^\omega, k) \right)
		\| < \delta
	\]
	the lemma follows directly from \eqref{eq:metric}. 
\end{proof}

We combine this with \eqref{eq:lyapunov} to obtain a backward stability for $(y_n, \eta_n)$. 

\begin{lemma}\label{lem:stab-conf}
Assume that $(y_n, \eta_n)$ satisfies $\Apm(\omega, \varphi, \delta, N)$ with 
\[
	\delta \ge e^{-(\lambda' - 3\epsilon)N/3}.
\]
There exists $C_1^\epsilon(\omega) >0$ with $e^{-\epsilon}\le C_1^\epsilon(\omega)/C_1^{\epsilon}(\theta\omega) \le  e^\epsilon$ such that, if  for $N > N_4(\omega)$,  minimizer , if  $- N/3 \le j < k \le 0$, satisfies $\|y_j - x_j^\omega\|<r^\epsilon(\theta^j \omega)$, $ \|y_k - x_k^\omega\| < r^\epsilon(\theta^k \omega)$, we have 
\[
	\|y_j - x_j^\omega\| \le C_1^\epsilon(\theta^k \omega) e^{\epsilon|j-k|} \left( \|y_k - x_k^\omega\| + 3 \sqrt{\delta}  \right). 
\]
\end{lemma}

\begin{proof}
	Apply Lemma~\ref{lem:xk-zk-comp}, we have
	\begin{align*}
		& Q^N(y_j, j) - Q^N(x_j^\omega, j)  \le Q^N(y_j, j) - Q^N(z_j^\omega, j) +  2 e^{-(\lambda' - 3\epsilon)N/3} \\
		& \le Q^N(y_k, k) - Q^N(z_k^\omega, k) + 2 e^{-(\lambda' - 3\epsilon)N/3}                                    \\& \le Q^N(y_k, k) - Q^N(x_k^\omega, k) + 4 e^{-(\lambda' - 3\epsilon)N/3} \\
		& \le Q^N(y_k, k) - Q^N(x_k^\omega, k)  + 4\delta.                                                           
	\end{align*}
	Combine with Lemma~\ref{lem:QN-metric}, we get 
	\[
		a^\epsilon(\theta^j \omega) \|y_j - x_j^\omega\|^2 \le K^\epsilon(\theta^k \omega) \|y_k - x_k^\omega\|^2  + 6\delta. 
	\]
	Using $a^\epsilon(\theta^j \omega) \ge e^{-\epsilon|j-k|}a^\epsilon(\theta^k \omega)$, we obtain
	\[
		\|y_j - x_j^\omega\|^2 \le K^\epsilon/a^\epsilon(\theta^k\omega) e^{\epsilon |j-k|} \left(  \|y_k - x_k^\omega\|^2 + 6\delta \right),
	\]
	therefore 
	\[
		\|y_j - x_j^\omega\|  \le e^{\epsilon |j-k|}  \sqrt{ K^\epsilon/a^\epsilon}(\theta^k \omega) \left(\|y_k - x_k^\omega\| + 3\sqrt{\delta} \right) .
	\]
	The lemma follows by taking $C_1^\epsilon = \sqrt{K^\epsilon/a^\epsilon}$.
\end{proof}

\section{Estimates from non-uniform hyperbolicity}
\label{sec:hyper-theory}

\subsection{Hyperbolic properties of the global minimizer} Denote 
\[
	X_n^\omega = (x_n^\omega, \zeta_n^\omega) \in \T^d \times \R^d,
\]
this orbit is non-uniformly hyperbolic.

\begin{proposition}\label{prop:nonuniform-hyp}
For any $\epsilon >0$, the following hold. 
\begin{enumerate}
	\item (Stable and unstable bundles) For each $n \in \Z$, there exists the splitting 
	\[
		\R^{2d} = E^s(X_n^\omega) \oplus E^u(X_n^\omega),
	\]
	where $\dim E^s = \dim E^u = d$. We denote by $\Pi^s_n, \Pi^u_n$ the projection to $E^s, E^u$ under this splitting. 
	
	\item (Lyapunov norm) There exist norms $\|\cdot\|_n^s$, $\|\cdot\|_n^u$ on $\R^d$, and the Lyapunov norm on $\R^{2d}$ is defined by 
	\[
		(\|v\|_n')^2 =   (\|\Pi_n^s v\|_n^s)^2 + ( \|\Pi_n^u v\|_n^u)^2.
	\]
	There exists a function $M^\epsilon(\omega)>0$ satisfying $e^{-\epsilon} \le M^\epsilon(\omega)/M^\epsilon(\theta \omega) \le e^\epsilon$ such that 
	\[
		\|v\| \le \|v\|_n' \le M^\epsilon(\theta^n \omega) \|v\|,
	\]
	where $\|\cdot\|$ is the Euclidean metric. We will omit the subscript from the $\|\cdot\|_n'$ and $\Pi^{s/u}_n$ when the index is clear from context. 
	
	\item (Cones) We define the unstable cones
	\[
		C_n^u = \{ v \in \R^{2d}: \quad  \|\Pi_n^s v\|_n^s \le \|\Pi_n^u v\|_n^u \},
	\]
	and the table cones 
	\[
		C_n^s = \{ v \in \R^{2d}: \quad  \|\Pi_n^u v\|_n^u \le \|\Pi_n^s v\|_n^s \}.
	\]
	
	\item (Hyperbolicity) There exists $\sigma^\epsilon(\omega)>0$ with $e^{-\epsilon} \le \sigma^{\epsilon}(\omega)/\sigma^\epsilon(\theta \omega) \le e^\epsilon$, such that the following hold. Let $Y_n$ be an orbit of $\Phi_n^\omega$. 
	\begin{enumerate}
		\item  If 
		\[
			\|Y_n - X_n\|'< \sigma^\epsilon(\theta^n \omega), \quad  \|Y_{n-1} - X_{n-1}\|' < \sigma^\epsilon(\theta^{n-1}\omega),  
		\]
		then   \[
			\|\Pi^s Y_{n-1} - \Pi^s X_{n-1}\|' \ge e^{\lambda'} \|\Pi^s Y_n - \Pi^s X_n\|',   
		\]
		where $\lambda' = \lambda - \epsilon$.  Moreover, if  $Y_n - X_n \in C_n^s$, then     $Y_{n-1} - X_{n-1} \in C_{n-1}^s$. 
		In other words, the stable cones are backward invariant and backward expanding. 
		\item If
		\[
			\|Y_n - X_n\|'< \sigma^\epsilon(\theta^n \omega), \quad  \|Y_{n-1} - X_{n-1}\|' < \sigma^\epsilon(\theta^{n-1}\omega),  
		\]
		then    
		\[
			\|\Pi^u Y_{n-1} - \Pi^u X_{n-1}\|' \le e^{-\lambda'} \|\Pi^u Y_n - \Pi^u X_n\|'. 
		\]
		Moreover, if  $Y_n - X_n \in C_n^u$,  then  $Y_{n-1} - X_{n-1} \in C_n^u$.
	\end{enumerate}
	
\end{enumerate}

\end{proposition}

\subsection{Stability of minimizer in the phase space}

We improve Lemma~\ref{lem:stab-conf} to its counter part in the phase space, using the Lyapunov norm. 

\begin{lemma}\label{lem:localize-phase} Under the same assumption of Lemma~\ref{lem:stab-conf}, there exists $C_2^\epsilon(\omega)>0$ and $e^\epsilon \le C_2^\epsilon(\omega)/C_2^\epsilon(\theta\omega) \le e^{-\epsilon}$ such that,  if $-N/3 < j < k \le 0$ satisfies 
\[
	\|y_j - x_j^\omega\| < r(\theta^j \omega), \quad \|y_k - x_k^\omega\| < r(\theta^k \omega), 
\]
then
\[
	\|Y_j - X_j^\omega\|' \le C_2^\epsilon(\theta^k \omega)e^{\epsilon|j-k|} \left(   \|y_k - x_k^\omega\| + \sqrt{\delta} \right). 
\]
\end{lemma}
\begin{proof}
	Apply Lemma~\ref{lem:stab-conf}, we have 
	\[
		\|y_j - x_j^\omega\| \le C_1^\epsilon(\theta^k \omega)  e^{\epsilon|j-k|} \left(  \|y_k - z_k^\omega\| + 3 \sqrt{\delta}  \right). 
	\]
	Since $j < -1$, we use Lemma~\ref{lem:grad-lip} to get 
	\[
		\|\eta_j - v_j^\omega\| \le K^{\epsilon}(\theta^j \omega) \|y_j - x_j^\omega\|, 
	\]
	and hence
	\[
		\|Y_j - X_j^\omega\| \le 2K^{\epsilon}(\theta^j \omega) \|y_j - x_j^\omega\|. 
	\]
	We now have 
	\begin{align*}
		& \|Y_j - X_j^\omega\|' \le M^\epsilon(\theta^j \omega)\|Y_j - X_j\| \le  2(M^\epsilon K^\epsilon)(\theta^j \omega) \|y_j - x_j^\omega\|                    \\
		& \le 2 C_1^\epsilon(\theta^k \omega) (M^\epsilon K^\epsilon)(\theta^k \omega) e^{2\epsilon|j-k|} \left(  \|y_k - x_k^\omega\| + 3\sqrt{\delta}  \right)    \\
		& \le 6 C_1^\epsilon(\theta^k \omega) (M^\epsilon K^\epsilon)(\theta^k \omega)e^{\epsilon|j-k|} \left(   \| y_k - x_k^\omega\|  + \sqrt{\delta}   \right) . 
	\end{align*}
	We now replace $\epsilon$ with $\epsilon/3$ in the above estimate, and define $C_2^\epsilon = 6C_1^{\epsilon/3}M^{\epsilon/3}K^{\epsilon/3}$, which satisfies $e^\epsilon \le C_2^\epsilon(\omega)/C_2^\epsilon(\theta\omega) \le e^{-\epsilon}$. The lemma follows. 
\end{proof}

\subsection{Exponential localization using hyperbolicity}

We now show hyperbolicity, together with Lemma~\ref{lem:localize-phase} lead to a stronger localization. In Lemma~\ref{lem:cone-contraction} we show a dichotomy: either $y_n - x_n^\omega$ contracts for each backward iterate, or $y_n$ is $\delta^q$ close to $x_n$ to begin with. 

Define 
\[
	r_1(\omega) = \min \{ r^{\epsilon/4}(\omega), \sigma^{\epsilon/4}(\omega)\},
\]
where $r^\epsilon$ is defined in Lemma~\ref{lem:tempered-arK}, and $\sigma^\epsilon$ defined in property (4) of Proposition~\ref{prop:nonuniform-hyp}. Now define 
\[
	\rho_0(\omega) = \left(  \frac{r_1(\omega)}{3 C_2^{\epsilon/4}(\omega)} \right)^2, 
\]
then $e^\epsilon \le \rho_0(\omega)/\rho_0(\theta\omega), r_1(\omega)/r_1(\theta\omega) \le e^\epsilon$.

\begin{lemma}\label{lem:cone-contraction}
Let $(y_n, \eta_n)_{n=-N}^0$ be the minimizer as before. Suppose for a given $-N/6 \le k \le -1$, we have 
\begin{equation}
	\label{eq:Yk-delta-initial-bound}
	\|Y_k - X_k\| < \rho_0(\theta^k \omega), \quad  \delta^{\frac14} < \rho_0(\theta^k \omega).
\end{equation}
Then one of the following alternatives must hold for $Y_k$: 
\begin{enumerate}
	\item $Y_k - X_k^\omega \in C_k^s$ and 
	\begin{equation}
		\label{eq:exp-delta}
		\|Y_k - X_k^\omega\|' \le  \max \{\delta^q, e^{-(\lambda' - 2\delta)N/6}\}, 
		\quad q = \lambda'/(8\varepsilon). 
	\end{equation}
	\item  $Y_k - X_k^\omega \in C_k^u$ and 
	\begin{equation}
		\label{eq:one-step-contraction}
		\|\Pi^u Y_{k-1} - \Pi^u X_{k-1}\|'  \le e^{-\lambda'} \| \Pi^u Y_k- \Pi^u X_k\|.     
	\end{equation}  
\end{enumerate}
\end{lemma}
\begin{proof}
	Let us denote $d_k = \| Y_k - X_k^\omega\|'$, $\bar{\delta} = \sqrt{\delta}$ and $C_3^\epsilon = C_2^{\epsilon/4}$. Define $-N/6 \le i_0 < 0$ by the relation 
	\begin{equation}
		\label{eq:i0}
		\begin{aligned}
			e^{2\epsilon|i_0|} & = \min \left\{ \frac{r_1(\theta^k \omega)}{ 3d_k C_3^{\epsilon}(\theta^k \omega)},\quad \frac{r_1(\theta^k \omega)}{3\delta_1 C_3^\epsilon(\theta^k \omega)} ,\quad  e^{\epsilon N/3} \right\} \\
			& = \min \left\{ \frac{\sqrt{\rho_0}(\theta^k \omega)}{ d_k}, \quad \frac{\sqrt{\rho_0}(\theta^k \omega)}{ \bar{\delta}}, \quad e^{\epsilon N/3} \right\}                                        
		\end{aligned}	  
	\end{equation}
	In particular, we have 
	\begin{equation}
		\label{eq:ei0-plus}
		C_3^\epsilon(\theta^k \omega) e^{2\epsilon|i_0|} \le \min \left\{  \frac{r_1(\theta^k\omega)}{3d_k}, \frac{r_1(\theta^k\omega)}{3\bar{\delta}}  \right\}
	\end{equation}
	and (using $C_3^\epsilon = \frac{r_1}{3 \sqrt{\rho}_0}$ and \eqref{eq:Yk-delta-initial-bound})
	\begin{equation}
		\label{eq:ei0-minus}
		C_3^\epsilon(\theta^k \omega) e^{-2\epsilon|i_0|} \le \max\left\{ 
		\frac{d_k r_1(\theta^k \omega)}{ 3 \rho_0(\theta^k \omega)}, \quad 
		\frac{\bar{\delta} r_1(\theta^k \omega)}{ 3 \rho_0(\theta^k \omega)} 
		\right\} < \frac{r_1(\theta^k \omega)}{3} < 1. 
	\end{equation}

	Since  $-N/3 \le i_0 + k <0$, Lemma~\ref{lem:localize-phase} applies.  Using \eqref{eq:ei0-plus}, we have for each $i_0 + k \le j < k$, 
	\begin{equation}
		\label{eq:dk-dj}
		\begin{aligned}
			& d_j \le C_3^\epsilon(\theta^k \omega) e^{\epsilon|j-k|} ( d_k + \bar{\delta})                                                                                                                                                       
			= e^{-\epsilon|i_0|} C_3^\epsilon(\theta^k \omega) e^{2\epsilon|i_0|} (d_k + \bar{\delta})  \\
			& \le e^{-\epsilon|i_0|} (d_k + \bar{\delta}) \min \left\{  \frac{r_1(\theta^k\omega)}{3d_k}, \frac{r_1(\theta^k\omega)}{3\bar{\delta}}  \right\}  \le e^{-\epsilon|i_0|} \cdot \frac23 r_1(\theta^k \omega)  < r_1(\theta^j \omega). 
		\end{aligned}  
	\end{equation}
	As a result, Proposition~\ref{prop:nonuniform-hyp} (4) applies for $i_0 + k \le j < k$. 

	If $Y_k - X_k^\omega \in C_k^u$ (second alternative), \eqref{eq:one-step-contraction} holds by  Proposition~\ref{prop:nonuniform-hyp} (4)(b). If $Y_k - X_k^\omega \in C^s_k$ (first alternative), then $Y_j - X_j^\omega \in C^s_j$ for all $i_0 + k < j < k$, due to backward invariance of stable cones. As a result, we get from Proposition~\ref{prop:nonuniform-hyp} (4)(a) that 
	\[
		d_j \ge e^{\lambda' |j-k|} d_k.  
	\]
	Pick $j = i_0 + k$, combine with the first line of \eqref{eq:dk-dj}, and using \eqref{eq:ei0-minus}, we get 
	\begin{align*}
		& d_k \le  C_3^\epsilon(\theta^k \omega) e^{-(\lambda'-\epsilon)|i_0|}(d_k + \bar{\delta}) 
		= C_3^\epsilon(\theta^k \omega) e^{-2\epsilon|i_0|} e^{-(\lambda'-3\epsilon)|j-k|}(d_k + \bar{\delta}) \\
		& \le  e^{-(\lambda'-3\epsilon)|i_0|}(d_k + \bar{\delta}).                                 
	\end{align*}
	We can choose $\epsilon$ small enough (and as a result $|i_0|$ large enough) such that $e^{-(\lambda'-3\epsilon)|i_0|}< \frac12$, then  
	\[
		\bar{\delta} e^{- 3(\lambda' - 3\epsilon)|i_0|} \ge d_k ( 1-  e^{- 3(\lambda' - 3\epsilon)|i_0|}) \ge \frac12 d_k. 
	\]
	Note that in this case, $d_k \le \bar\delta$.  As a result, using \eqref{eq:i0}, and $\bar{\delta}^{\frac12} < \rho(\theta^k\omega)$ from \eqref{eq:Yk-delta-initial-bound}, we get 
	\[
		e^{-2\epsilon|i_0|} = \max \left\{ \frac{\bar{\delta}}{\sqrt{\rho_0(\theta^k\omega)}} ,\quad  e^{-\epsilon N/6} \right\} =   \max \{ \bar{\delta}^{\frac12}, \quad  e^{-\epsilon N/3}\} = \max\{\delta^{\frac14}, \quad  e^{-\epsilon N/3}\}.
	\]
	We combine with the previous formula to get 
	\[
		d_k \le 2 \bar{\delta} \left( e^{-2\epsilon|i_0|} \right)^{(\lambda' - 3\epsilon)/(2\epsilon)} \le \left( e^{-2\epsilon|i_0|} \right)^{(\lambda' - 3\epsilon)/(2\epsilon)} \le    \max\left\{   \delta^q, \quad e^{-(\lambda' -3\epsilon)N/6}  \right\} 
	\]
	with $q = (\lambda' - 3\epsilon)/(8\epsilon)$. 
\end{proof}

We are now ready to prove Proposition~\ref{prop:dt-localize}. 
\begin{proof}[Proof of Proposition~\ref{prop:dt-localize}]
	Define 
	\[
		\rho(\omega) = \rho_0(\omega)/(4C_3^\epsilon(\omega)) < \rho_0(\omega), 
	\] 
	then $e^{-2\epsilon}\le \rho(\theta \omega)/\rho(\omega) \le e^{2\epsilon}$. Recall the assumption 
	\[
		\delta^{\frac18} < \rho(\omega). 
	\]
	Denote $k_0 = \max\left\{  \frac{1}{8\epsilon} \log \delta, - \frac{N}{6}  \right\}$, 
	we have 
	\[
		e^{-\epsilon|k_0|} = e^{\epsilon k_0} \ge e^{\frac18 \log \delta} = \delta^{\frac18}
	\]
	and
	\[
		\rho_0(\theta^k \omega) \ge e^{-\epsilon |k_0|} \rho_0(\omega) \ge \delta^{\frac14}, \quad k_0 \ge k \ge 0. 
	\]
	If $\|y_0 - x_0^\omega\| <\rho(\omega)$, by Lemma~\ref{lem:localize-phase}, we have 
	\[
		\|Y_{-1} - X_{-1}^\omega\|' \le 
		C_3^\epsilon(\omega) e^\epsilon \left(  \|y_0 - x_0^\omega\| + \sqrt{\delta} \right)
		\le 
		C_3^\epsilon (\omega) e^{\epsilon}  2 \rho(\omega) = e^{\epsilon} \rho_0(\omega)/2. 
	\]
	Assume $e^{\epsilon} < 2$, then Lemma~\ref{lem:cone-contraction} applies for $k = -1$. If alternative \eqref{eq:exp-delta} hold, we obtain $\|Y_k - X_k^\omega\|' \le  \max \{\delta^q, e^{-(\lambda' - 3\epsilon)N/6}\}$ and the proposition follows.  Otherwise, the alternative \eqref{eq:one-step-contraction} applies, and 
	\[
		\begin{aligned}
			& \|Y_{-2} - X_{-2}^\omega\|' \le \sqrt{2} \|\Pi^s(Y_{-2} - X_{-2}^\omega)\|^s \le \sqrt{2} e^{-\lambda'} \|Y_{-1} - X_{-1}\|' \\
			& <  e^{-\lambda'} e^{\epsilon}/\sqrt{2} \rho_0(\epsilon) <  e^{-\lambda'} \rho_0(\omega) \le  \rho_0(\theta^{-1}\omega) .     
		\end{aligned}
	\]
	if $e^\epsilon < \sqrt{2}$. 
	We can apply Lemma~\ref{lem:cone-contraction} again. Suppose  \eqref{eq:exp-delta} does not hold for $k+1 \le j \le -1$. Then 
	\[
		\|Y_k - X_k^\omega\|' \le \sqrt{2} e^{-|k+1|\lambda'} \|Y_{-1}- X_{-1}\|' < e^{-|k+1|\lambda'}\rho_0(\omega) \le \rho_0(\theta^k \omega). 
	\]
	Therefore this argument can be applied inductively until we reach $k = k_0$, then 
	\[
		\begin{aligned}
			& \|y_{k_0} - x_{k_0}^\omega\| \le \|Y_{k_0}- X_{k_0}^\omega\|' \le \sqrt{2} e^{-\lambda'|k_0|} \le \sqrt{2}\rho_0(\omega) \max\left\{ \delta^{\frac{\lambda'}{8\epsilon}}, e^{-\lambda' N/6}    \right\} \\
			& < \max \{\delta^q, e^{-(\lambda' - 3\epsilon)N/6}\}.                                                                                                                                                    
		\end{aligned}
	\]
\end{proof}

\bibliographystyle{plain}
\bibliography{HJ}
\end{document}